\documentclass[12pt,onecolumn]{IEEEtran}
\usepackage{amsmath,amssymb,euscript ,yfonts,psfrag,latexsym,dsfont,graphicx,bbm,color,amstext,wasysym,subfig,parskip}
\graphicspath{{./},{./figures/}}

\begin{document}
\newtheorem{thm}{Theorem}
\newtheorem{cor}[thm]{Corollary}
\newtheorem{conj}[thm]{Conjecture}
\newtheorem{lemma}[thm]{Lemma}
\newtheorem{prop}{Proposition}
\newtheorem{problem}[thm]{Problem}
\newtheorem{remark}[thm]{Remark}
\newtheorem{defn}[thm]{Definition}
\newtheorem{ex}[thm]{Example}

\newcommand{\nrho}{\rho}
\newcommand{\mR}{{\mathbb R}}
\newcommand{\R}{{\mathbb R}}
\newcommand{\mD}{{\mathbb D}}
\newcommand{\mE}{{\mathbb E}}  
\newcommand{\cN}{{\mathcal N}}
\newcommand{\cR}{{\mathcal R}}
\newcommand{\cS}{{\mathcal S}}
\newcommand{\cC}{{\mathcal C}}
\newcommand{\cD}{{\mathcal D}}
\newcommand{\cE}{{\mathcal E}}
\newcommand{\cF}{{\mathcal F}}
\newcommand{\cK}{{\mathcal K}}
\newcommand{\cL}{{\mathcal L}}
\newcommand{\cU}{{\mathcal U}}
\newcommand{\cV}{{\mathcal V}}
\newcommand{\cP}{{\mathcal P}}
\newcommand{\cH}{{\mathcal H}}
\newcommand{\cQ}{{\mathcal Q}}
\newcommand{\diag}{\operatorname{diag}}
\newcommand{\tr}{\operatorname{trace}}
\newcommand{\rank}{\operatorname{rank}}
\newcommand{\f}{{\mathfrak f}}
\newcommand{\g}{{\mathfrak g}}
\newcommand{\range}{\cR}  
\newcommand{\trace}{\operatorname{trace}}
\newcommand{\argmin}{\operatorname{argmin}}
\renewcommand{\break}{\vspace*{.2in}\hrule\vspace*{.2in}}

\newcommand{\xx}{x}
\newcommand{\x}{x}
\newcommand{\E}{{\mathbb E}}
\newcommand{\D}{{\mathbb D}}
\newcommand{\vv}{v}
\newcommand{\tstart}{0}
\newcommand{\tend}{1}
\newtheorem{proposition}{Proposition}

\newcommand{\ignore}[1]{}

\def\spacingset#1{\def\baselinestretch{#1}\small\normalsize}
\setlength{\parskip}{10pt}
\setlength{\parindent}{20pt}
\spacingset{1}

\newcommand{\mike}{\color{magenta}}
\definecolor{grey}{rgb}{0.4,0.4,0.4}
\definecolor{lightgray}{rgb}{0.97,.99,0.99}

\title{Optimal transport over a linear dynamical system}


\author{Yongxin Chen, Tryphon Georgiou and Michele Pavon}

\maketitle
{\begin{abstract}
We consider the problem of steering an initial probability density for the state vector of a linear system
to a final one, in finite time, using minimum energy control. In the case where the dynamics correspond to an integrator ($\dot x(t) = u(t)$) this amounts to a Monge-Kantorovich Optimal Mass Transport (OMT) problem. In general, we show that the problem can again be reduced to solving an OMT problem and that it has a unique solution. In parallel, we study the optimal steering of the state-density of a linear stochastic system with white noise disturbance; this is known to correspond to a {\em Schr\"odinger bridge}. As the white noise intensity tends to zero, the flow of densities converges to that of the deterministic dynamics and can serve as a way to compute the solution of its deterministic counterpart. The solution can be expressed in closed-form for Gaussian initial and final state densities in both cases.
\end{abstract}}


\section{Introduction}
We are interested in stochastic control problems to steer the probability density of the state-vector of a linear system between an initial and a final distribution for two cases, i) with and ii) without stochastic disturbance. That is, we consider the linear dynamics
\begin{equation}\label{eq:definingsystem}
dx(t)=A(t)x(t)dt+B(t)u(t)dt+ \sqrt{\epsilon}B(t)dw(t)
\end{equation}
where $w$ is a Wiener process, $u$ is a control input, $x$ is the state process, and $(A,B)$ is a controllable pair of matrices, for the two cases where i) $\epsilon>0$ and ii) $\epsilon=0$. In either case, the state is a random vector with an initial distribution $\mu_0$. Our task is to determine a minimum energy input that drives the system to a final state distribution $\mu_1$ over the time interval\footnote{There is no loss in generality having time window $[0,1]$ instead of, the more general $[t_0,t_1]$. This is done for notational convenience.} $[0,\,1]$, that is, the minimum of
\begin{equation}\label{eq:energy}
\E\{\int_0^1\|u(t)\|^2dt\}
\end{equation}
subject to $\mu_1$ being the probability distribution of the state vector at $t=1$.

When the state distribution represents density of particles whose position obeys $\dot x(t) = u(t)$ (i.e., $A(t)\equiv 0$, $B(t)\equiv I$, and $\epsilon=0$) the problem reduces to the classical Optimal Mass Transport ({\bf OMT}) problem\footnote{Historically, the modern formulation of OMT is due to Leonid Kantorovich \cite{Kan42} and has been the focus of dramatic developments because of its relevance in many diverse fields including economics, physics, engineering, and probability \cite{Rac98,Vil03,Vil08,GanMcc96,JorKinOtt98,BenBre00,AmbGigSav06,Vil03,Vil08,NinGeoTan13}. Kantorovich's contributions and the impact of the OMT to resource allocation was recognized with the Nobel Prize in Economics in 1975.}
 with quadratic cost \cite{Vil03,BenBre00}. Thus, the above problem, for $\epsilon=0$, represents a generalization of OMT to deal with particles obeying known ``prior'' {\em non-trivial} dynamics while being steered between two end-point distributions -- we refer to this as the problem of {\em OMT with prior dynamics} ({\bf OMT-wpd}). The problem of OMT-wpd was first introduced in our previous work \cite{CheGeoPav14e} for the case where $B(t)\equiv I$. The difference of course to the classical OMT is that, here, the linear dynamics are arbitrary and may facilitate or hinder transport. Applications are envisioned in the steering of particle beams through time-varying potential, the steering of swarms (UAV's, large collection of microsatelites, etc.), as well as in the modeling of the flow and collective motion of particles, clouds, platoons, flocking of insects, birds, fish, etc.\ between end-point distributions  \cite{LeoFio01}, and the interpolation/morphing of distributions \cite{AngHakTan03}.

In the case where $\epsilon>0$ and a stochastic disturbance is present, the flow of ``particles'' is dictated by dynamics as well as by Brownian diffusion. The corresponding
stochastic control problem to steer the state density function between the end-point distributions has been recently shown to be equivalent to the so-called {\em Schr\"odinger bridge problem}\footnote{The Schr\"odinger bridge problem, in its original formulation \cite{Sch31,Sch32,Wak90}, seeks a probability law on  path space with given two end-point marginals which is close to a Markovian {\em prior} distribution in the sense of large deviations (minimum relative entropy).
Early important contributions were due to Fortet, Beurling, Jamison and F\"{o}llmer \cite{For40,Beu60,Jam74,Fol88} while renewed interest was sparked after a close relationship to stochastic control was recognized \cite{Dai91,DaiPav90,PavWak91}.} \cite{Dai91,CheGeoPav14a,CheGeoPav14d}.
The Sch\"odinger bridge problem can be seen as a stochastic version of OMT due to the presence of the diffusive term in the dynamics. As a result, its solution is more well behaved due to the smoothing property of the Laplacian. On the other hand, it follows from \cite{Leo12,Leo13,Mik04, MikThi08} that for the special case $A(t)\equiv 0$ and $B(t)\equiv I$, the solution to the Schr\"odinger bridge problem tends to that of the OMT when ``slowing down'' the diffusion by taking $\epsilon\to 0$. These two facts suggest the Schr\"odinger bridge problem as a means to construct solutions to OMT for both, the standard one as well as the problem of OMT with prior dynamics.

The present work begins with an expository prologue on OMT (Section \ref{sec:OMT}). We then develop the theory of OMT-wpd (Section \ref{sec:OMTwithprior}) and establish that OMT-wpd always has a unique solution.
Next we discuss in parallel the theory of the Sch\"odinger bridge problem for linear dynamics and arbitrary end-point marginals (Section \ref{sec:slowingdown}). We focus on the connection between the two problems and in Theorem \ref{thm:priorslowingdown} we establish that the solution to the OMT-wpd is indeed the limit, in a suitable sense, of the corresponding solution to the Schr\"odinger bridge problem. In Section \ref{sec:gaussian} we specialize to the case of linear dynamics with Gaussian marginals, where closed-form solutions are available for both problems. The form of solution underscores the connection between the two and that the OMT-wpd is the limit of the Schr\"odinger bridge problem when the diffusion term vanishes. In Section \ref{sec:examples} we work out two academic examples to highlight the relation between the two problems (OMT and Schr\"odinger bridge).

\section{Optimal mass transport}\label{sec:OMT}

Consider two nonnegative measures $\mu_0, \mu_1$ on ${\mathbb R}^n$ having equal total mass.  These may represent probability distributions, distribution of resources, etc.
In the original formulation of OMT, due to Gaspar Monge, a transport (measurable) map
\[
T\;:\;{\mathbb R}^n\to{\mathbb R}^n\;:\;x\mapsto T(x)
\]
is sought that specifies where mass $\mu_0(dx)$ at $x$ must be transported so as to match the final distribution in the sense that $T_\sharp \mu_0=\mu_1$, i.e.
$\mu_1$ is the ``push-forward'' of $\mu_0$ under $T$ meaning
\[
\mu_1(B)=\mu_0(T^{-1}(B))
\]
for every Borel set in $\R^n$. Moreover, the map must incur minimum cost of transportation
\[
\int c(x,T(x))\mu_0(dx).
\]
Here, $c(x,y)$ represents the transportation cost per unit mass from point $x$ to point $y$ and in this section it will be taken as $c(x,y)=\frac{1}{2}\|x-y\|^2$.

The dependence of the transportation cost on $T$ is highly nonlinear and a minimum may not exist. This fact complicated early analyses to the problem due to Abel and others \cite{Vil03}. A new chapter opened in 1942 when Leonid Kantorovich presented a relaxed formulation. In this, instead of seeking a transport map, we seek a joint distribution $\Pi(\mu_0,\mu_1)$ on the product space $\R^n\times\R^n$ so that the marginals along the two coordinate directions coincide with $\mu_0$ and $\mu_1$ respectively.
The joint distribution $\Pi(\mu_0,\mu_1)$ is refered to as ``coupling" of $\mu_0$ and $\mu_1$.
Thus, in the Kantorovich formulation we seek
    \begin{equation}\label{eq:OptTrans}
        \inf_{\pi\in\Pi(\mu_0,\mu_1)}\int_{\R^n\times\R^n}\frac{1}{2}\|x-y\|^2d\pi(x,y)
    \end{equation}
When the optimal Monge-map $T$ exists, the support of the coupling is precisely the graph of $T$, see~\cite{Vil03}.

Formulation \eqref{eq:OptTrans} represents a ``static'' end-point formulation, i.e., focusing on ``what goes where''.
Ingenious insights due to Benamou and Brenier \cite{BenBre00} and \cite{McC97}
led to a fluid dynamic formulation of OMT. An elementary derivation of the above was presented in \cite{CheGeoPav14e} which we now follow. OMT is first cast as a stochastic control problem with atypical boundary constraints:
    \begin{subequations}\label{eq:stochcontr}
    \begin{eqnarray}\label{eq:stochcontr1}
        &&\inf_{\vv\in\cV}\E\left\{\int_{\tstart}^{\tend}\frac{1}{2}\|\vv(t,\x^\vv(t))\|^2dt\right\}
        \\
        && \dot{\x}^\vv(t)=\vv(t,\x^\vv(t)),
        \\
        && \x^v(\tstart)\sim\mu_0,\quad \x^v(\tend)\sim\mu_1.
    \end{eqnarray}
    \end{subequations}
Here $\cV$ represents the family of continuous feedback control laws.
From this point on we assume that
$\mu_0$ and $\mu_1$ are absolutely continuous, i.e., $\mu_0(dx)=\rho_0(x)dx$, $\mu_1(dy)=\rho_1(y)dy$ with $\rho_0,\rho_1$ corresponding density functions, and accordingly a distribution for $\x^\vv(t)\sim\rho(t,x)dx$. Then, $\rho$ satisfies weakly\footnote{In the sense that, $\int_{\mR^n\times[0,1]}(\partial f/\partial t+v\cdot\nabla f)\rho dxdt=0$ for smooth functions $f$ with compact support.} the continuity equation
    \begin{equation}\label{eq:continuity}
        \frac{\partial \rho}{\partial t}+\nabla\cdot(\vv\rho)=0
    \end{equation}
expressing the conservation of probability mass and
    \[
        \E\left\{\int_{\tstart}^{\tend}\frac{1}{2}\|\vv(t,\x^\vv(t))\|^2dt\right\}=\int_{\R^n}\int_{\tstart}^{\tend}\frac{1}{2}\|\vv(t,x)\|^2\rho(t,x)dtdx.
    \]
As a consequence, (\ref{eq:stochcontr}) is reformulated as a ``fluid-dynamics'' problem \cite{BenBre00}:
    \begin{subequations}\label{eq:BB}
    \begin{eqnarray}\label{eq:BB1}
        &&\inf_{(\rho,v)}\int_{\R^n}\int_{\tstart}^{\tend}\frac{1}{2}\|\vv(t,x)\|^2\rho(t,x)dtdx,\\&&\frac{\partial \rho}{\partial t}+\nabla\cdot(\vv\rho)=0,\label{eq:BB2}\\&& \rho(\tstart,x)=\rho_0(x), \quad \rho(\tend,y)=\rho_1(y).\label{eq:boundary}
    \end{eqnarray}
    \end{subequations}

\subsection{Solutions to OMT}

Assuming that $\mu_0,\mu_1$ are absolutely continuous ($d\mu_0(dx)=\rho_0(x)dx$ and $d\mu_1(dx)=\rho_1(x)dx$) it is a standard result that OMT has a unique solution \cite{Bre91,Vil03,Vil08} and that an optimal transport $T$ map exists and is the gradient of a convex function $\phi$, i.e.,
	\begin{equation}\label{eq:optimalmap}
		y=T(x)=\nabla\phi(x).
	\end{equation}
By virtue of the fact that the push-forward of $\mu_0$ under $\nabla\phi$ is $\mu_1$, this function satisfies a particular case of the  Monge-Amp\`ere equation \cite[p.126]{Vil03}, \cite[p.377]{BenBre00}, namely, $\det(H\phi(x))\rho_1(\nabla \phi(x))=\rho_0(x)$, where $H\phi$ is the Hessian matrix of $\phi$, which is a fully nonlinear second-order elliptic equation. The computation of $\phi$ has received attention only recently \cite{BenBre00}, \cite{AngHakTan03} where numerical schemes have been developed. We will appeal to the availability of $\phi$ in the sequel without being concerned about its explicit computation.
	
Having $T$, the displacement of the mass along the path from $t=0$ to $t=1$ is
    \begin{subequations}\label{eq:displacementinterp}
	\begin{equation}
		\mu_t=(T_t)_\sharp \mu_0,~~~T_t(x)=(1-t)x+tT(x)
	\end{equation}
while $\mu_t$ is absolutely continuous with derivative
	\begin{equation}
		\rho(t,x)=d\mu_t(x)/dx.
	\end{equation}
    \end{subequations}
	Then, $\vv(t,x)=T\circ T_t^{-1}(x)-T_t^{-1}(x)$ and $\rho(t,x)$ together solve \eqref{eq:BB}. Here $\circ$ denotes the composition of maps.

\subsection{Variational analysis}\label{sec:var1}

In this subsection we briefly recapitulate the sufficient optimality conditions for a pair $(\rho(\cdot,\cdot),v(\cdot,\cdot))$ to be a solution of \eqref{eq:BB} from \cite{BenBre00} (see also \cite[Section II]{CheGeoPav14e} for an alternative elementary derivation).
 \begin{proposition}
Consider $\rho^*(t,x)$ with $t\in[\tstart,\tend]$ and $x\in {\mathbb R}^n$, that satisfies
    \begin{subequations}\label{eq:sufficient}
    \begin{equation}\label{eq:optev}
        \frac{\partial \rho^*}{\partial t}+\nabla\cdot(\nabla\psi\rho^*)=0, \quad \rho^*(\tstart,x)=\rho_0(x),
    \end{equation}
where $\psi$ is a solution of the Hamilton-Jacobi equation
    \begin{equation}\label{eq:HJclass}
        \frac{\partial \psi}{\partial t}+\frac{1}{2}\|\nabla\psi\|^2=0.
    \end{equation}
If in addition
    \begin{equation}
        \rho^*(\tend,x)=\rho_1(x),
    \end{equation}
    \end{subequations}
then the pair $\left(\rho^*,\vv^*\right)$ with $\vv^*(t,x)=\nabla\psi(t,x)$ is a solution of \eqref{eq:BB}.
\end{proposition}

The stochastic nature of the Benamou-Brenier formulation \eqref{eq:BB} stems from the fact that initial and final densities are specified. Accordingly, the above requires solving a two-point boundary value problem and the resulting control dictates the local velocity field. In general, one cannot expect to have a classical solution of \eqref{eq:HJclass} and has to be content with a viscosity solution. Let $\psi$ be a viscosity solution of \eqref{eq:HJclass} that admits the Hopf-Lax representation \cite[p.\ 174]{Vil03}\cite[p.\ 4]{Leo13}
    \[
        \psi(t,x)=\inf_y \left\{\psi(0,y)+\frac{\|x-y\|^2}{2t}\right\},~~t\in(0,1]
    \]
with
    \[
        \psi(0,x)=\phi(x)-\frac{1}{2}\|x\|^2
    \]
and $\phi$ as in \eqref{eq:optimalmap}, then this $\psi$ together with the displacement interpolation $\rho$ in \eqref{eq:displacementinterp} is a solution to~\eqref{eq:sufficient}.

\section{Optimal mass transport with prior dynamics}\label{sec:OMTwithprior}

Optimal transport has also been studied for general cost $c(x,y)$ that
derives from an action functional
\begin{equation}\label{cost}
c(x,y)=\inf_{\x\in\mathcal X_{xy}}\int_{\tstart}^{\tend}L(t,x(t),\dot{x}(t))dt,
\end{equation}
where the Lagrangian $L(t,x,p)$ is strictly convex and superlinear in the velocity variable $p$, see \cite[Chapter 7]{Vil08}, \cite[Chapter 1]{Fig08}, \cite{BerBuf07}. Existence and uniqueness of an optimal transport map $T$ has been established\footnote{OMT has also been studied and similar results established for $\R^n$ replaced by a Riemannian manifold.} for general cost functionals as in  \eqref{cost}.
It is easy to see that the choice $c(x,y)=\frac{1}{2}\|x-y\|^2$ is the special case where
\[L(t,x,p)=\frac{1}{2}\|p\|^2.\]
Another interesting special case is when
\begin{equation}\label{special}\quad L(t,x,p)=\frac{1}{2}\|p-v(t,x)\|^2.
\end{equation}
This has been motivated by a transport problem ``with prior" associated to the velocity field $v(t,x)$ \cite[Section VII]{CheGeoPav14e}.  There the prior was thought to reflect a solution to a ``nearby'' problem that needs to be adjusted so as to be consistent with updated estimates for marginals.

An alternative motivation for \eqref{special} is to address transport in an ambient flow field $v(t,x)$. In this case, assuming the control has the ability to steer particles in all directions, transport will be effected according to dynamics
\[
\dot x(t)=v(t,x)+u(t)
\]
where $u(t)$ represents control effort and 
\[
\int_0^1\frac{1}{2}\|u(t)\|^2dt=\int_0^1\frac{1}{2}\|\dot{x}(t)-v(t,x)\|^2dt
\]
represents corresponding quadratic cost (energy). 
Thus, it is of interest to consider more general dynamics where the control does not affect directly all state directions.  One such example is the problem to steer inertial particles in phase space through force input (see \cite{CheGeoPav14a} and \cite{CheGeoPav14b} where similar problems have been considered for dynamical systems with stochastic excitation).

Therefore, herein, we consider a natural generalization of OMT where the transport paths are required to satisfy dynamical constraints. We focus our attention on linear dynamics and, consequently, cost of the form
\begin{subequations}\label{generalcontrol}
\begin{eqnarray}c(x,y)&=&\inf_{\mathcal U}\int_{\tstart}^{\tend}\tilde L(t,x(t),u(t))dt,\mbox{ where}\\&&\dot{x}(t)=A(t)x(t)+B(t)u(t),\\ &&x(0)=x,\quad x(1)=y,
\end{eqnarray}
\end{subequations}
and $\mathcal U$ is a suitable class of controls\footnote{Note that we use a common convention to denote by $x$ a point in the state space and by $x(t)$ a state trajectory.}. This formulation extends the transportation problem in a similar manner as optimal control  generalizes the classical calculus of variations \cite{FleRis75} (albeit herein only for linear dynamics). 
It is easy to see that \eqref{special} corresponds to  $A(t)=0$ and $B(t)$ the identity matrix in \eqref{generalcontrol}. When $B(t)$ is invertible, \eqref{generalcontrol} reduces to \eqref{cost} by a change of variables, taking
\[
L(t,x,p)=\tilde L(t,x,B(t)^{-1}(p-A(t)x)).
\]
 However,
when $B(t)$ is not invertible, an analogous change of variables leads to a Lagrangian $L(t,x,p)$ that fails to satisfy the classical conditions (strict convexity and superlinearity in $p$). Therefore, in this case, the existence and uniqueness of an optimal transport map $T$ has to be established independently. We do this for the case where $\tilde L(t,x,u)=\|u\|^2/2$ corresponding to power.

We now formulate the corresponding stochastic control problem. The system dynamics
	\begin{equation}\label{eq:priorsys}
	 	\dot x(t)=A(t)x(t)+B(t)u(t),
	\end{equation}
are assumed to be controllable and the initial state $x(0)$ to be a random vector with probability density $\rho_0$. We seek a minimum energy continuous feedback control law $u(t,x)$ that steers the system
to a final state $x(1)$ having distribution $\rho_1(x)dx$. That is, we address the following:
    \begin{subequations}\label{eq:priorstochastic}
    \begin{eqnarray}
        && \inf_{u\in \cU} \mE \left\{\int_{\tstart}^{\tend}\frac{1}{2}\|u(t,x^u)\|^2dt\right\},
        \\
        && \dot x^u(t)=A(t)x^u(t)+B(t)u(t),
        \\
        && x^u(0) \sim \mu_0, \quad x^u(1) \sim \mu_1,
    \end{eqnarray}
    \end{subequations}
where $\cU$ is the family of continuous feedback control laws. Once again, this can be recast in a ``fluid-dynamics'' version in terms of the sought one-time probability density functions of the state vector:
    \begin{subequations}\label{eq:priorBB}
    \begin{eqnarray}\label{eq:priorBB1}
        && \inf_{(\nrho,u)} \int_{\mR^n}\int_{\tstart}^{\tend}\frac{1}{2}\|u(t,x)\|^2\nrho(t,x)dtdx,
        \\
        && \frac{\partial \nrho}{\partial t}+\nabla\cdot((A(t)x+B(t)u)\nrho)=0,\label{eq:priorBB2}
        \\
        && \nrho(\tstart,x)=\rho_0(x), \quad \nrho(\tend,y)=\rho_1(y).\label{eq:priorboundary}
    \end{eqnarray}
    \end{subequations}
Naturally, for the trivial prior dynamics $A(t)\equiv 0$ and $B(t)\equiv I$, the problem reduces to the classical OMT  and the solution $\{\nrho(t,\cdot)\mid 0\le t\le 1\}$ is the {\em displacement interpolation} of the two marginals \cite{McC97}.
In the rest of the section, we show directly that Problem \eqref{eq:priorBB} has a unique solution. 

\subsection{Solutions to OMT-wpd}

Let $\Phi(t_1,t_0)$ be the state transition matrix of \eqref{eq:priorsys} from $t_0$ to $t_1$, with $\Phi_{10}:=\Phi(1,0)$, and
    \[
        M_{10}:=M(\tend,\tstart)=\int_{\tstart}^{\tend}\Phi(\tend,t)B(t)B(t)'\Phi(\tend,t)'dt
    \]
be the controllability Gramian of the system. Recall \cite{LeeMar67,AthFal66} 
 that for linear dynamics \eqref{eq:priorsys} and given boundary conditions $x(0)=x$, $x(1)=y$,
the least energy and the corresponding control input can be given in closed-form, namely
    \begin{equation}\label{eq:minimumenergy}
        \inf \int_0^1\frac{1}{2}\|u(t)\|^2dt= \frac{1}{2}(y-\Phi_{10}x)'M_{10}^{-1}(y-\Phi_{10}x)
    \end{equation}
which is attained for
    \[
        u(t)=B(t)'\Phi(1,t)'M_{10}^{-1}(y-\Phi_{10}x),
    \]
and the corresponding optimal trajectory
    \begin{equation}\label{eq:priorgeodesic}
        x(t)=\Phi(t,1)M(1,t)M_{10}^{-1}\Phi_{10}x+M(t,0)\Phi(1,t)'M_{10}^{-1}y.
    \end{equation}
Problem \eqref{eq:priorstochastic} can now be written as
    \begin{subequations}\label{eq:priorBBB}
    \begin{eqnarray}
        && \label{eq:priorBBB1} \inf_\pi \int_{\mR^n\times\mR^n}\frac{1}{2}(y-\Phi_{10}x)'M_{10}^{-1}(y-\Phi_{10}x) d\pi(x,y)
        \\
        && \label{eq:priorBBB2}
        \int_{dx}\int_{y\in \mR^n} d\pi(x,y)=\rho_0(x)dx,~~\int_{dy}\int_{x\in \mR^n} d\pi(x,y)=\rho_1(y)dy,
    \end{eqnarray}
    \end{subequations}
where $\pi$ is a measure on $\mR^n\times \mR^n$.

Problem \eqref{eq:priorBBB} can be converted to the Kantorovich formulation \eqref{eq:OptTrans} of the OMT by a transformation of coordinates. Indeed, consider the linear map
    \begin{equation}\label{eq:coordtrans}
        C: \left[\begin{array}{c} x \\y \end{array}\right]\longrightarrow
        \left[\begin{array}{c} \hat{x} \\\hat{y} \end{array}\right]=\left[\begin{array}{c} M_{10}^{-1/2}\Phi_{10}x \\M_{10}^{-1/2}y \end{array}\right]
    \end{equation}
and set
\[
    \hat\pi=C_{\sharp}\pi.
\]
Clearly, (\ref{eq:priorBBB1}-\ref{eq:priorBBB2}) become
    \begin{subequations}\label{eq:priorBBBB}
    \begin{equation}
        \inf_{\hat\pi} \int_{\mR^n\times\mR^n}\frac{1}{2}\|\hat{y}-\hat{x}\|^2d\hat\pi(\hat{x},\hat{y})
    \end{equation}
    \begin{equation}
        \int_{d\hat{x}}\int_{\hat{y}\in \mR^n} d\hat\pi(\hat{x},\hat{y})=\hat\rho_0(\hat{x})d\hat{x},~~
        \int_{d\hat{y}}\int_{\hat{x}\in \mR^n} d\hat\pi(\hat{x},\hat{y})=\hat\rho_1(\hat{y})d\hat{y},
    \end{equation}
    \end{subequations}
where
    \begin{eqnarray*}
        \hat\rho_0(\hat{x})&=&|M_{10}|^{1/2}|\Phi_{10}|^{-1}\rho_0(\Phi_{10}^{-1}M_{10}^{1/2}\hat{x})\\
        \hat\rho_1(\hat{y})&=&|M_{10}|^{1/2}\rho_1(M_{10}^{1/2}\hat{y}).
    \end{eqnarray*}
Problem \eqref{eq:priorBBBB} is now a standard  OMT with quadratic cost function and we know that the optimal transport map $\hat T$ for this problem exists.  It is the gradient of a convex function $\phi$, i.e.,
    \begin{equation}\label{eq:priorphi}
        \hat{T}=\nabla \phi,
    \end{equation}
and  the optimal $\hat\pi$ is concentrated on the graph of $\hat{T}$  \cite{Bre91}. The solution to the original problem \eqref{eq:priorBBBB} can now be determined using $\hat T$, and it is
	\begin{equation}\label{eq:monotoneforoneD}
		y=T(x)=M_{10}^{1/2}\hat T(M_{10}^{-1/2}\Phi_{10}x).
	\end{equation}
The one-time marginals can be readily computed as the push-forward
    \begin{subequations}\label{eq:priorinterpolation}
    \begin{eqnarray}
        \mu_t&=&(T_t)_\sharp \mu_0,
         \end{eqnarray}
        where
            \begin{eqnarray}
        T_t(x)&=&\Phi(t,1)M(1,t)M_{10}^{-1}\Phi_{10}x+M(t,0)\Phi(1,t)'M_{10}^{-1}T(x),
    \end{eqnarray}
and
    \begin{equation}\label{eq:priordisplaceinterp}
        \nrho(t,x)=d\mu_t(x)/dx.
    \end{equation}
    \end{subequations}
In this case, we refer to the parametric family of one-time marginals as {\em displacement interpolation with prior dynamics}.

\subsection{Variational analysis}\label{sec:variational_omtwp}
In this section we present a variational analysis for the OMT-wpd \eqref{eq:priorBB} analogous to that for the OMT problem \cite{BenBre00}, \cite[Section II]{CheGeoPav14e}. The analysis provides conditions for a pair $(\rho(\cdot,\cdot),v(\cdot,\cdot))$ to be a solution to OMT-wpd and will be used in Section \ref{sec:gaussian} to prove optimality of the solution of the OMT-wpd with Gaussian marginals.

Let $\cP_{\rho_0\rho_1}$ be the family of flows of probability densities
satisfying the boundary conditions and $\cU$ be the family of continuous feedback control laws $u(\cdot,\cdot)$. Consider the unconstrained minimization over $\cP_{\rho_0\rho_1}\times\cU$ of the Lagrangian
    \begin{equation}\label{eq:priorlagrangian}
        \mathcal L(\nrho,u,\lambda)=
        \int_{\mR^n}\int_{\tstart}^{\tend}\left[\frac{1}{2}\|u(t,x)\|^2\nrho(t,x)+\lambda(t,x)\left(\frac{\partial \nrho}{\partial t}+\nabla\cdot((A(t)x+B(t)u)\nrho)\right)\right]dtdx,
    \end{equation}
where $\lambda$ is a $C^1$ Lagrange multiplier. After integration by parts, assuming that limits for $x\rightarrow\infty$ are zero, and observing that the boundary values are constant over $\mathcal P_{\rho_0\rho_1}$, we get the problem
    \begin{equation}\label{eq:priorlagrangian2}
        \inf_{(\nrho,u)\in\cP_{\rho_0\rho_1}
        \times\cU}\int_{\R^n}\int_{\tstart}^{\tend}\left[\frac{1}{2}\|u(t,x)\|^2+\left(-\frac{\partial \lambda}{\partial t}-\nabla\lambda\cdot (A(t)x+B(t)u)\right)\right]\nrho(t,x)dtdx.
    \end{equation}
Pointwise minimization with respect to $u$ for each fixed flow of probability densities $\nrho$ gives
    \begin{equation}\label{eq:prioroptcond}
        u^*(t,x)=B(t)'\nabla\lambda(t,x).
    \end{equation}
Substituting into (\ref{eq:priorlagrangian2}), we get
    \begin{equation}\label{eq:priorlagrangian3}
        J(\nrho,\lambda)=-\int_{\R^n}\int_{\tstart}^{\tend}\left[\frac{\partial \lambda}{\partial t}+A(t)x\cdot\nabla\lambda+\frac{1}{2}\nabla\lambda\cdot B(t)B(t)'\nabla\lambda\right]\nrho(t,x)dtdx.
    \end{equation}
As in Section \ref{sec:var1}, we get the following sufficient conditions for optimality:
\begin{proposition}\label{prop:sufficientconditions}
Consider $\nrho^*$ that satisfies
    \begin{subequations}\label{eq:priorsufficient}
    \begin{equation}\label{eq:prioroptev}
        \frac{\partial \nrho^*}{\partial t}+\nabla\cdot[(A(t)x+B(t)B(t)'\nabla\psi)\nrho^*]=0, \quad \nrho^*(\tstart,x)=\rho_0(x),
    \end{equation}
where $\psi$ is a solution of the Hamilton-Jacobi equation
    \begin{equation}\label{eq:hamiltonjacobi}
    \frac{\partial \psi}{\partial t}+x'A(t)'\nabla\psi+\frac{1}{2}\nabla\psi' B(t)B(t)'\nabla\psi=0.
    \end{equation}
If in addition
    \begin{equation}\label{eq:terminalcondition}
        \nrho^*(\tend,x)=\rho_1(x),
    \end{equation}
    \end{subequations}
then the pair $\left(\nrho^*(t,x),u^*(t,x)=B(t)'\nabla\psi(t,x)\right)$ is a solution to problem~\eqref{eq:priorBB}.
\end{proposition}

It turns out that \eqref{eq:priorsufficient} always admits a solution. In fact, one solution can be constructed as follows.
\begin{proposition}
Given dynamics \eqref{eq:priorsys} and marginal distributions $\mu_0(dx)=\rho_0(x)dx, \mu_1(dx)=\rho_1(x)dx$, let $\psi(t,x)$ be defined by the formula
    \begin{equation}\label{eq:priorpsi}
        \psi(t,x)=\inf_y \left\{\psi(0,y)+\frac{1}{2}(x-\Phi(t,0)y)'M(t,0)^{-1}(x-\Phi(t,0)y)\right\}
    \end{equation}
with
    \begin{equation*}
        \psi(0,x)=\phi(M_{10}^{-1/2}\Phi_{10}x)-\frac{1}{2}x'\Phi_{10}'M_{10}^{-1}\Phi_{10}x
    \end{equation*}
and $\phi$ as in \eqref{eq:priorphi}. Moreover, let $\rho(t,\cdot)=(T_t)_\sharp \rho_0$ be the displacement interpolation as in \eqref{eq:priorinterpolation}. Then this pair $(\psi, \rho)$ is a solution to \eqref{eq:priorsufficient}.
\end{proposition}
\begin{proof}
First we show that \eqref{eq:priorpsi} satisfies \eqref{eq:hamiltonjacobi}. Let $H(t,x,\nabla\psi)$ be the Hamiltonian of the Hamilton-Jacobi equation \eqref{eq:hamiltonjacobi}, that is,
    \[
        H(t,x,\nabla \psi)=x'A(t)'\nabla \psi+\frac{1}{2}\nabla \psi'B(t)B(t)'\nabla \psi,
    \]
and define
    \begin{eqnarray*}
        L(t,x,v)&=&\sup_{p} \{p\cdot v-H(t,x,p)\}\\
        &=&
        \begin{cases}
        {\frac{1}{2}(v-A(t)x)'(B(t)B(t)')^\dagger (v-A(t)x)} & \text{if}~ v-A(t)x\in\range(B)\\
        {\infty} & \text{otherwise},
        \end{cases}
    \end{eqnarray*}
where $\dagger$ denotes pseudo-inverse and $\range(\cdot)$ denotes ``the range of''. Then the {\em Bellman principle of optimality} \cite{FleRis75} yields a particular solution of \eqref{eq:hamiltonjacobi}
    \begin{eqnarray*}
        \psi(t,x)&=&\inf_y \left\{\psi_0(y)+\int_0^t L(\tau,\xi(\tau),\dot\xi(\tau)),~~\xi(t)=x, \xi(0)=y\right\}\\
        &=& \inf_y \left\{\psi_0(y)+\frac{1}{2}(x-\Phi(t,0)y)'M(t,0)^{-1}(x-\Phi(t,0)y)\right\}.
    \end{eqnarray*}
This shows that \eqref{eq:priorpsi} is indeed a solution of \eqref{eq:hamiltonjacobi}.

Next we show $(\psi, \rho)$ is a (weak) solution of \eqref{eq:prioroptev}. Define
    \[
        v(t,x)=A(t)x+B(t)B(t)'\nabla\psi(t,x),
    \]
then \eqref{eq:prioroptev} becomes a linear transport equation
    \begin{equation}\label{eq:lineartransport}
        \frac{\partial \rho}{\partial t}+\nabla\cdot (v\rho)=0,
    \end{equation}
with velocity field $v(t,x)$.
We claim 
    \[
        v(t,\cdot)\circ T_t=dT_t/dt,
    \] 
that is, $v(t,x)$ is the velocity field associated with the trajectories $(T_t)$. If this claim is true, then the linear transport equation \eqref{eq:lineartransport} follows from a standard argument \cite[p.\ 167]{Vil03}. Moreover, the terminal condition \eqref{eq:terminalcondition} follows since $\rho(1,\cdot)=T_\sharp \rho_0$.
We next prove the claim.
Formula \eqref{eq:priorpsi} can be rewritten as
    \[
        g(x)=\sup_y\left\{x'M(t,0)^{-1}\Phi(t,0)y-f(y)\right\},
    \]
with
    \begin{eqnarray*}
        g(x) &=& \frac{1}{2}x'M(t,0)^{-1}x-\psi(t,x)\\
        f(y) &=& \frac{1}{2}y'\Phi(t,0)'M(t,0)^{-1}\Phi(t,0)y+\psi(0,y).
    \end{eqnarray*}
The function
    \begin{eqnarray*}
        f(y)&=&\frac{1}{2}y'\Phi(t,0)'M(t,0)^{-1}\Phi(t,0)y+\psi(0,y)\\
        &=&\frac{1}{2}y'\left[\Phi(t,0)'M(t,0)^{-1}\Phi(t,0)-\Phi_{10}'M_{10}^{-1}\Phi_{10}\right]y+\phi(M_{10}^{-1/2}\Phi_{10}y)
    \end{eqnarray*}
is convex since $\phi$ is convex and the matrix
    \begin{eqnarray*}
        \Phi(t,0)'M(t,0)^{-1}\Phi(t,0)-\Phi_{10}'M_{10}^{-1}\Phi_{10}
        &=&
        \left(\int_0^t\Phi(0,\tau)B(\tau)B(\tau)'\Phi(0,\tau)'d\tau\right)^{-1}\\
        &&-\left(\int_0^1\Phi(0,\tau)B(\tau)B(\tau)'\Phi(0,\tau)'d\tau\right)^{-1}
    \end{eqnarray*}
is positive semi-definite. Hence, from a similar argument to the case of Legendre transform, we obtain
    \[
        \nabla g\circ(M(t,0)\Phi(0,t)'\nabla f)=M(t,0)^{-1}\Phi(t,0).
    \]
It follows
    \[
        (M(t,0)^{-1}-\nabla\psi(t,\cdot))\circ\left\{M(t,0)\Phi(0,t)'\left[\Phi(t,0)'M(t,0)^{-1}\Phi(t,0)x
        +\nabla\psi(0,x)\right]\right\}=M(t,0)^{-1}\Phi(t,0)x.
    \]
After some cancellations it yields
    \[
        \nabla\psi(t,\cdot)\circ \Phi(t,0)x+\nabla\psi(t,\cdot)\circ M(t,0)\Phi(0,t)'\nabla\psi(0,x)-\Phi(0,t)'\nabla\psi(0,x)=0.
    \]
On the other hand, since
    \[
        T(x)=M_{10}^{-1/2}\nabla\phi(M_{10}^{-1/2}\Phi_{10}x)
        =M_{10}\Phi_{01}'\nabla\psi(0,x)+\Phi_{10}x,
    \]
we have 
    \begin{eqnarray*}
        T_t(x)&=&\Phi(t,1)M(1,t)M_{10}^{-1}\Phi_{10}x+M(t,0)\Phi(1,t)'M_{10}^{-1}T(x)
        \\&=&
        \Phi(t,0)x+M(t,0)\Phi(0,t)'\nabla\psi(0,x),
    \end{eqnarray*}
from which it follows 
    \[
        \frac{dT_t(x)}{dt}=A(t)\Phi(t,0)x+A(t)M(t,0)\Phi(0,t)'\nabla\psi(0,x)+B(t)B(t)'\Phi(0,t)'\nabla\psi(0,x).
    \]
Therefore,
    \begin{eqnarray*}
        v(t,\cdot)\circ T_t(x)-\frac{dT_t(x)}{dt}
        &=&
        \left[A(t)+B(t)B(t)'\nabla\psi(t,\cdot)\right]\circ\left[\Phi(t,0)x+M(t,0)\Phi(0,t)'\nabla\psi(0,x)\right]
        \\&&-\left[A(t)\Phi(t,0)x+A(t)M(t,0)\Phi(0,t)'\nabla\psi(0,x)+B(t)B(t)'\Phi(0,t)'\nabla\psi(0,x)\right]
        \\&=& B(t)B(t)'\left\{\nabla\psi(t,\cdot)\circ \Phi(t,0)x+\nabla\psi(t,\cdot)\circ M(t,0)\Phi(0,t)'\nabla\psi(0,x)\right.
        \\&& \left.-\Phi(0,t)'\nabla\psi(0,y)\right\}
        \\&=&0,
    \end{eqnarray*}
which completes the proof.
\end{proof}

\section{Schr\"{o}dinger bridges and their zero-noise limit}\label{sec:slowingdown}

In 1931/32, Schr\"odinger \cite{Sch31,Sch32} treated the following problem: A large number N of i.i.d. Brownian particles in $\mR^n$ is observed to have at time $t=0$ an empirical distribution approximately equal to $\rho_0(x)dx$, and at some later time $t=1$ an empirical distribution approximately equal to $\rho_1(x)dx$. Suppose that $\rho_1(x)$ considerably differs from what it should be according to the law of large numbers, namely
    \[
        \int q^B(0,x,1,y)\rho_0(x)dx,
    \]
where
    \[
        q^B(s,x,t,y)=(2\pi)^{-n/2}(t-s)^{-n/2}\exp\left(-\frac{\|x-y\|^2}{2(t-s)}\right)
    \]
denotes the Brownian transition probability density. It is apparent that the particles have been transported in an unlikely way. But of the many unlikely ways in which this could have happened, which one is the most likely? The process that is consistent with the observed marginals and fulfils Schr\"odinger's requirement is referred to as the Schr\"odinger bridge.

This problem has a long history \cite{Wak90}. In particular, F\"ollmer \cite{Fol88} showed that the solution to Schr\"odinger's problem corresponds to a probability law $\cP^B$  on path space that minimizes the relative entropy with respect to the Wiener measure among all laws with given initial and terminal distributions, $\rho_0(x)dx$ and $\rho_1(x)dx$, respectively, and proved that the minimizer always exists.
Beurling \cite{Beu60} and Jamison \cite{Jam74} generalized the idea of the Schr\"odinger bridge by changing the Wiener measure to a more general reference measure induced by a Markov process. Jamison's result is stated below.

\begin{thm}\label{thm:schrodinger}
Given two probability measures $\mu_0(dx)=\rho_0(x)dx$ and $\mu_1(dx)=\rho_1(x)dx$ on $\mR^n$ and the continuous, everywhere positive Markov kernel $q(s,x,t,y)$, there exists a unique pair of $\sigma$-finite measure $(\hat\varphi_0(x) dx,\varphi_1(x) dx)$ on $\mR^n$ such that the measure $\cP_{01}$ on $\mR^n\times\mR^n$ defined by
    \begin{equation}\label{eq:distributionjoint}
        \cP_{01}(E)=\int_E q(0,x,1,y)\hat\varphi_0(x)\varphi_1(y)dxdy
    \end{equation}
has marginals $\mu_0$ and $\mu_1$. Furthermore, the Schr\"odinger bridge from $\mu_0$ to $\mu_1$ is determined via the distribution flow
    \begin{subequations}\label{eq:distributionflow}
    \begin{equation}\label{eq:distributionflow1}
        \cP_t(dx)=\varphi(t,x)\hat{\varphi}(t,x)dx
    \end{equation}
with
    \begin{eqnarray}
        \varphi(t,x)&=&\int q(t,x,1,y)\varphi_1(y)dy\\
        \hat{\varphi}(t,x)&=& \int q(0,y,t,x) \hat{\varphi}_0(y)dy.
    \end{eqnarray}
    \end{subequations}
\end{thm}

The flow \eqref{eq:distributionflow} is referred to as the {\em entropic interpolation with prior $q$} between $\mu_0$ and $\mu_1$, or simply entropic interpolation, when it is clear what the Markov kernel $q$ is.
An efficient numerical algorithm to obtain the pair 
 $(\hat\varphi_0,\varphi_1)$ and thereby solve the Schr\"odinger bridge problem is given in \cite{CheGeoPav15a}.

For the case of non-degenerate Markov processes, a connection between the Schr\"odinger problem and stochastic optimal control was drawn by Dai Pra \cite{Dai91}. In particular, for the case of a Brownian kernel, he showed that
the one-time marginals $\rho(t,x)$ for Schr\"odinger's problem can be obtained as solutions to
    \begin{subequations}\label{eq:daipra}
    \begin{eqnarray}\label{eq:expectedcost}
        && \inf_{(\nrho,v)} \int_{\mR^n}\int_{\tstart}^{\tend}\frac{1}{2}\|v(t,x)\|^2\nrho(t,x)dtdx,
        \\
        && \frac{\partial \nrho}{\partial t}+\nabla\cdot(v\nrho)- \label{eq:fokkerplanck}
        \frac{1}{2}\Delta \nrho=0,
        \\
        && \nrho(\tstart,x)=\rho_0(x), \quad \nrho(\tend,y)=\rho_1(y).
    \end{eqnarray}
    \end{subequations}
Here, \eqref{eq:expectedcost} is the infimum of the expected cost while \eqref{eq:fokkerplanck} is the corresponding Fokker-Planck equation. The entropic interpolation is $ \cP_t(dx)=\rho(t,x)dx$.

An alternative equivalent reformulation given in \cite{CheGeoPav14e} is\begin{subequations}\label{eq:SBB}
    \begin{eqnarray}\label{SBB1}&&\inf_{(\rho,v)}\int_{\R^n}\int_{\tstart}^{\tend}\left[\frac{1}{2}\|v(x,t)\|^2+\frac{1}{8}\|\nabla\log\rho(x,t)\|^2\right]\rho(t,x)dtdx,\\&&\frac{\partial \rho}{\partial t}+\nabla\cdot(v\rho)=0,\label{SBB2}\\&& \rho(\tstart,x)=\rho_0(x), \quad \rho(\tend,y)=\rho_1(y),\label{Sboundary}
\end{eqnarray}
\end{subequations}
where the Laplacian in the dynamical constraint is traded for a ``Fisher information'' regularization term in the cost functional. Although the form in \eqref{eq:SBB} is quite appealing, for the purposes of this paper we will use only \eqref{eq:daipra}.
    
Formulation \eqref{eq:daipra} is quite similar to OMT \eqref{eq:BB} except for the presence of the Laplacian in \eqref{eq:fokkerplanck}.
It has been shown \cite{Mik04,MikThi08,Leo12,Leo13} that the OMT problem is, in a suitable sense, indeed the limit of the Schr\"{o}dinger problem when the diffusion coefficient of the reference Brownian motion goes to zero. In particular, the minimizers of the Schr\"{o}dinger problems converge to the unique solution of OMT as explained below.

\begin{thm}\label{thm:slowingdown}
 Given two probability measures $\mu_0(dx)=\rho_0(x)dx, \mu_1(dx)=\rho_1(x)dx$ on $\mR^n$ with finite second moment, let $\cP_{01}^{B,\epsilon}$ be the solution of the Schr\"{o}dinger problem with Markov kernel
    \begin{equation}\label{eq:browniankernel}
        q^{B,\epsilon}(s,x,t,y)=(2\pi)^{-n/2}((t-s)\epsilon)^{-n/2}\exp\left(-\frac{\|x-y\|^2}{2(t-s)\epsilon}\right)
    \end{equation}
and marginals $\mu_0, \mu_1$, and let $\cP_t^{B,\epsilon}$ be the corresponding entropic interpolation.
Similarly, let $\pi$ be the solution to the OMT problem \eqref{eq:OptTrans} with the same marginal distributions, and $\mu_t$ the corresponding displacement interpolation.  Then,  $\cP_{01}^{B,\epsilon}$ converges
weakly\footnote{A sequence $\{P_n\}$ of probability measures on a metric space ${\cal S}$ converges weakly to a measure $P$ if $\int_{\cal S}fdP_n\rightarrow\int_{\cal S}fdP$ for every bounded, continuous function $f$ on the space.}
to $\pi$ and $\cP_t^{B,\epsilon}$  converges weakly  to $\mu_t$, as $\epsilon$ goes to $0$.
\end{thm}

To build some intuition on the relation between OMT and Schr\"odinger bridges, consider
    \[
        dx(t)=\sqrt{\epsilon}dw(t)
    \]
with $w(t)$ being the standard Wiener process; the Markov kernel of $x(t)$ is $q^{B,\epsilon}$ in \eqref{eq:browniankernel}. The corresponding Schr\"odinger bridge problem with the law of $x(t)$ as prior, is equivalent to
    \begin{subequations}\label{eq:slowingdownbridge}
    \begin{eqnarray}\label{eq:slowingdownbridge1}
        && \inf_{(\nrho,v)} \int_{\mR^n}\int_{\tstart}^{\tend}\frac{1}{2\epsilon}\|v(t,x)\|^2\nrho(t,x)dtdx,
        \\
        && \frac{\partial \nrho}{\partial t}+\nabla\cdot(v\nrho)-
        \frac{\epsilon}{2}\Delta \nrho=0,\label{eq:slowingdownbridge2}
        \\
        && \nrho(\tstart,x)=\rho_0(x), \quad \nrho(\tend,y)=\rho_1(y).
    \end{eqnarray}
    \end{subequations}
Note that the solution exists for all $\epsilon$ and
coincides with the solution of the problem to minimize the cost functional
\[
\int_{\mR^n}\int_{\tstart}^{\tend}\frac{1}{2}\|v(t,x)\|^2\nrho(t,x)dtdx
\]
instead, i.e., ``rescaling'' \eqref{eq:slowingdownbridge1} by removing the factor $1/\epsilon$.
Now observe that the only difference between  \eqref{eq:slowingdownbridge} after removing the scaling $1/\epsilon$ in the cost functional and the OMT formulation \eqref{eq:BB} is the regularization term $\frac{\epsilon}{2}\Delta \nrho$ in \eqref{eq:slowingdownbridge2}. Thus, formally, the constraint \eqref{eq:slowingdownbridge2} becomes \eqref{eq:BB2} as $\epsilon$ goes to 0.
Below we discuss a general result that includes the case when the zero-noise limit of Schr\"odinger bridges corresponds to OMT with (linear) dynamics. This problem has been studied in \cite{Leo12} in a more abstract setting based on Large Deviation Theory \cite{DemZei09}. Here we consider the special case that is connected to our OMT-wpd formulation.
To this end, we begin with the Markov kernel corresponding to the process
    \[
        dx(t)=A(t)x(t)dt+\sqrt{\epsilon}B(t)dw(t).
    \]
The entropic interpolation $\cP_t(dx)=\nrho(t,x)dx$ can be obtained by solving (the ``rescaled'' problem)
    \begin{subequations}\label{eq:priorschrodingerbridge}
    \begin{eqnarray}\label{eq:priorschrodingerbridge1}
        && \inf_{(\nrho,u)} \int_{\mR^n}\int_{\tstart}^{\tend}\frac{1}{2}\|u(t,x)\|^2\nrho(t,x)dtdx,
        \\
        && \frac{\partial \nrho}{\partial t}+\nabla\cdot((A(t)x+B(t)u)\nrho)-
        \frac{\epsilon}{2}\sum_{i,j=1}^n\frac{\partial^2 (a(t)_{ij}\nrho)}{\partial x_i \partial x_j}=0,\label{eq:priorschrodingerbridge2}
        \\
        && \nrho(\tstart,x)=\rho_0(x), \quad \nrho(\tend,y)=\rho_1(y),\label{eq:priorboundarybridge}
    \end{eqnarray}
    \end{subequations}
where $a(t)=B(t)B(t)'$, see \cite{CheGeoPav14c,CheGeoPav14d}. This result represents a slight generalization of Dai Pra's result \cite{Dai91} in that the stochastic differential equation corresponding to (\ref{eq:priorschrodingerbridge2}) may be degenerate (i.e., $\rank(a(t))\neq n$).
Comparing \eqref{eq:priorschrodingerbridge} with \eqref{eq:priorBB} we see that the only difference is the extra term
    \[
        \frac{\epsilon}{2}\sum_{i,j=1}^n\frac{\partial^2 (a(t)_{ij}\nrho)}{\partial x_i \partial x_j}
    \]
in \eqref{eq:priorschrodingerbridge2} as compared to \eqref{eq:priorBB2}. Formally, \eqref{eq:priorschrodingerbridge2} converges to \eqref{eq:priorBB2} as $\epsilon$ goes to 0. This suggests that the minimizer of the OMT-wpd might be obtained as the limit of the joint initial-final time distribution of solutions to the Schr\"odinger bridge problems as the diffusivity goes to zero. This result is stated next and can be proved based on the result in \cite{Leo12} together with the Freidlin-Wentzell Theory \cite[Section 5.6]{DemZei09} (a large deviation principle on sample path space). In the Appendix, we also provide a direct proof which doesn't require a large deviation principle.

\begin{thm}\label{thm:priorslowingdown}
Given two probability measures $\mu_0(dx)=\rho_0(x)dx, \mu_1(dx)=\rho_1(x)dx$ on $\mR^n$ with finite second moment, let $\cP_{01}^\epsilon$ be the solution of the Schr\"{o}dinger problem with reference Markov process
    \begin{equation}\label{eq:priorprocess}
        dx(t)=A(t)x(t)dt+\sqrt{\epsilon}B(t)dw(t)
    \end{equation}
and marginals $\mu_0, \mu_1$, and let $\cP_t^\epsilon$ be the corresponding entropic interpolation. Similarly, let $\pi$ be the solution to \eqref{eq:priorBBB} with the same marginal distributions, and $\mu_t$ the corresponding displacement interpolation. Then, $\cP_{01}^\epsilon$ converges weakly to $\pi$ and $\cP_t^\epsilon$ converges weakly to $\mu_t$ as $\epsilon$ goes to $0$.
\end{thm}

An important consequence of this theorem is that we can now use the numerical algorithm in \cite{CheGeoPav15a} which provides a solution to the Schr\"odinger problem, for a vanishing $\epsilon$, as a means to solve the general problem of OMT with prior dynamics (and, in particular, the standard OMT \cite{CheGeoPav15a}). This is highlighted in the examples of Section \ref{sec:examples}. It should be noted that the algorithm, which relies on computing the pair
 $(\hat\varphi_0,\varphi_1)$ in Theorem \ref{thm:schrodinger}, is totally different from other numerical algorithms that solve standard OMT problems \cite{BenBre00}, \cite{AngHakTan03}.
 
\section{Gaussian marginals}\label{sec:gaussian}
We now consider the correspondence between Schr\"{o}dinger bridges and OMT-wpd for the special case where the marginals are normal distributions. That the OMT-wpd solution corresponds to the zero-noise limit of the Schr\"odinger bridges is of course a consequence of Theorem \ref{thm:priorslowingdown}, but in this case, we can obtain explicit expressions in closed-form and this is the point of this section.

Consider the reference evolution
	\begin{equation}\label{eq:dynamicslinear}
		d\xx(t)=A(t)\xx(t)dt+\sqrt{\epsilon}B(t)dw(t)
	\end{equation}
and the two marginals 
	\begin{subequations}\label{eq:marginalslinear}
	\begin{equation}
		\rho_0(x)=
        \frac{1}{\sqrt{(2\pi)^n|\Sigma_0|}}
        \exp\left[-\frac{1}{2}(x-m_0)'\Sigma_0^{-1}(x-m_0)\right],
	\end{equation} 		
	\begin{equation}
		\rho_1(x)=
        \frac{1}{\sqrt{(2\pi)^n|\Sigma_1|}}\exp\left[-\frac{1}{2}(x-m_1)'\Sigma_1^{-1}(x-m_1)\right],
	\end{equation}
	\end{subequations}
where, as usual, the system with matrices $(A(t), B(t))$ is controllable. 
In our previous work \cite{CheGeoPav14a,CheGeoPav14b}, we derived a ``closed-form'' expression for the Schr\"{o}dinger bridge when $m_0=m_1=0$, namely,
	\begin{equation}\label{eq:schrodingerbridgelinear1}
		d\xx(t)=(A(t)-B(t)B(t)'\Pi_\epsilon(t))\xx(t) dt+\sqrt{\epsilon}B(t)dw(t)
	\end{equation}
with $\Pi_\epsilon(t)$ satisfying the matrix Riccati equation
	\begin{equation}\label{eq:schrodingerbridgefeedback}
		\dot{\Pi}_\epsilon(t)+A(t)'\Pi_\epsilon(t)+\Pi_\epsilon(t)A(t)
        -\Pi_\epsilon(t)B(t)B(t)'\Pi_\epsilon(t)=0
	\end{equation}
and the boundary condition
	\begin{equation}\label{eq:boundarycondition}
        \Pi_\epsilon(0)=\Sigma_0^{-1/2}[\frac{\epsilon}{2}I+
	    \Sigma_0^{1/2}\Phi_{10}^\prime M_{10}^{-1}\Phi_{10}\Sigma_0^{1/2}-
	   (\frac{\epsilon^2}{4}I+
        \Sigma_0^{1/2}\Phi_{10}^\prime M_{10}^{-1}\Sigma_1M_{10}^{-1}
	    \Phi_{10}\Sigma_0^{1/2})^{1/2}]\Sigma_0^{-1/2}.
    \end{equation}
When $m_0\neq 0$ or  $m_1\neq 0$ the bridge becomes:
	\begin{equation}\label{eq:schrodingerbridgelinear2}
		d\xx(t)=(A(t)-B(t)B(t)'\Pi_\epsilon(t))\xx(t) dt
        +B(t)B(t)'m(t)dt+\sqrt{\epsilon}B(t)dw(t)
	\end{equation}
where
	\begin{equation}\label{eq:schrodingerbridgedrift}
        m(t)=
        \hat{\Phi}(\tend,t)'\hat{M}(\tend,\tstart)^{-1}(m_1-\hat{\Phi}(\tend,\tstart)m_0)
	\end{equation}
with $\hat{\Phi}(t,s), \hat{M}(t,s)$ satisfying
	\begin{equation*}
		\frac{\partial \hat{\Phi}(t,s)}{\partial t} =
        (A(t)-B(t)B(t)'\Pi_\epsilon(t))\hat{\Phi}(t,s),~~~\hat{\Phi}(t,t)=I
	\end{equation*}
and
	\[
		\hat{M}(t,s)=\int_s^t\hat{\Phi}(t,\tau)B(t)B(t)'\hat{\Phi}(t,\tau)'d\tau.
	\]

Next we consider the zero-noise limit
by letting $\epsilon$ go to $0$. In the case where $A(t)\equiv 0, B(t)\equiv I$, the Schr\"{o}dinger bridges converge to the solution of the OMT. In general, when $A(t)\not\equiv 0, B(t)\not\equiv I$, by taking $\epsilon= 0$ in \eqref{eq:boundarycondition} we obtain
	\begin{equation}\label{eq:schrodingerbridgeinitial}
		\Pi_0(0)=
        \Sigma_0^{-1/2}[\Sigma_0^{1/2}\Phi_{10}^\prime M_{10}							^{-1}\Phi_{10}\Sigma_0^{1/2}-(\Sigma_0^{1/2}\Phi_{10}^\prime M_{10}		^{-1}\Sigma_1M_{10}^{-1}\Phi_{10}\Sigma_0^{1/2})^{1/2}]\Sigma_0^{-1/2},
	\end{equation}
and the corresponding limiting process
	\begin{equation}\label{eq:optimaltransportprior}
		dx(t)=(A(t)-B(t)B(t)'\Pi_0(t))x(t) dt+B(t)B(t)'m(t)dt,
        ~~x(\tstart)\sim (m_0,\Sigma_0)
	\end{equation}
with $\Pi_0(t), m(t)$ satisfying \eqref{eq:schrodingerbridgefeedback}, \eqref{eq:schrodingerbridgedrift} and \eqref{eq:schrodingerbridgeinitial}. In fact $\Pi_0(t)$ has the explicit expression
    \begin{eqnarray}
        \nonumber
        \Pi_0(t)&=&-M(t,\tstart)^{-1}-M(t,\tstart)^{-1}\Phi(t,\tstart)\left[\Phi_{10}^\prime M_{10}^{-1}\Phi_{10}
        -\Sigma_0^{-1/2}(\Sigma_0^{1/2}\Phi_{10}^\prime M_{10}^{-1}\Sigma_1M_{10}^{-1}\Phi_{10}\Sigma_0^{1/2})^{1/2}\Sigma_0^{-1/2}\right.
        \\
        &&
        \left.-
        \Phi(t,\tstart)'M(t,\tstart)^{-1}\Phi(t,\tstart)\right]^{-1}\Phi(t,\tstart)'M(t,\tstart)^{-1}.
        \label{eq:feedbackPi}
    \end{eqnarray}

As indicated earlier, Theorem \ref{thm:priorslowingdown} already implies that
 \eqref{eq:optimaltransportprior} yields an optimal solution to \eqref{eq:priorBB}. Here we give an alternative proof by directly verifying that the corresponding displacement interpolation and the control satisfy the conditions of Proposition \ref{prop:sufficientconditions}.

\begin{proposition}
Let $\nrho(t,\cdot)$ be the probability density of the process $x(t)$ in \eqref{eq:optimaltransportprior}, and
	\[
	    u(t,x)=-B(t)'\Pi_0(t)x+B(t)'m(t),
	\]
then the pair $(\nrho,u)$ is a solution of the problem \eqref{eq:priorBB} with prior dynamics \eqref{eq:priorsys} and marginals \eqref{eq:marginalslinear}.
\end{proposition}
\begin{proof}
To prove that the pair $(\nrho,u)$ is a solution, we show first that $\nrho$ satisfies the boundary condition $\nrho(\tend,x)=\rho_1(x)$, and second, that $u(t,x)=B(t)'\nabla \psi(t,x)$ for some $\psi$ that satisfies the Hamilton-Jacobi equation~\eqref{eq:hamiltonjacobi}.

Equation \eqref{eq:optimaltransportprior} is linear with gaussian initial condition, hence $x(t)$ is a gaussian process. We claim that density of $x(t)$ is
    \[
        \nrho(t,x)=\frac{1}{\sqrt{(2\pi)^n|\Sigma(t)|}}\exp\left[-\frac{1}{2}(x-n(t))'\Sigma(t)^{-1}(x-n(t))\right],
    \]
where
    \[
        n(t)=\hat{\Phi}(t,\tstart)m_0+\int_{\tstart}^{t}\hat{\Phi}(t,\tau)B(\tau)B(\tau)'m(\tau)d\tau
    \]
and
    \begin{eqnarray}
        \nonumber
        \Sigma(t)&=&M(t,\tstart)\Phi(\tstart,t)'\Sigma_0^{-1/2}\left[-\Sigma_0^{1/2}\Phi_{10}^\prime M_{10}^{-1}\Phi_{10}\Sigma_0^{1/2}+(\Sigma_0^{1/2}\Phi_{10}^\prime M_{10}^{-1}\Sigma_1M_{10}^{-1}\Phi_{10}\Sigma_0^{1/2})^{1/2}\right.\\
        &&\left.+\Sigma_0^{1/2}\Phi(t,\tstart)'
        M(t,\tstart)^{-1}\Phi(t,\tstart)\Sigma_0^{1/2}\right]^2 \Sigma_0^{-1/2}\Phi(\tstart,t)M(t,\tstart)
        \label{eq:statecovariance}
    \end{eqnarray}
for $t\in(\tstart,\tend]$. It is obvious that $\E\{x(t)\}=n(t)$ and it is also immediate that
    \[
        \lim_{t\to \tstart}\Sigma(t)=\Sigma_0.
    \]
  Straightforward but lengthy computations show that
  $\Sigma(t)$ satisfies the Lyapunov differential equation
    \[
        \dot{\Sigma}(t)=(A(t)-B(t)B(t)'\Pi_0(t))\Sigma(t)+\Sigma(t)(A(t)-B(t)B(t)'\Pi_0(t))'.
    \]
    Hence, $\Sigma(t)$ is the covariance of $x(t)$.
Now, observing that
    \begin{eqnarray*}
        n(\tend)&=&\hat{\Phi}(\tend,\tstart)m_0+\int_{\tstart}^{\tend}\hat{\Phi}(\tend,\tau)B(\tau)B(\tau)'m(\tau)d\tau\\
        &=&\hat{\Phi}(\tend,\tstart)m_0+\int_{\tstart}^{\tend}\hat{\Phi}(\tend,\tau)B(\tau)B(\tau)'\hat{\Phi}(\tend,\tau)'
        d\tau\hat{M}(\tend,\tstart)^{-1}(m_1-\hat{\Phi}(\tend,\tstart)m_0)=m_1
    \end{eqnarray*}
and
    \begin{eqnarray*}
        \Sigma(\tend)&=&M(\tend,\tstart)\Phi(\tstart,\tend)'\Sigma_0^{-1/2}\left[(\Sigma_0^{1/2}\Phi_{10}^\prime M_{10}^{-1}\Sigma_1M_{10}^{-1}\Phi_{10}\Sigma_0^{1/2})^{1/2}\right]^2\Sigma_0^{-1/2}\Phi(\tstart,\tend)M(\tend,\tstart)=\Sigma_1,
    \end{eqnarray*}
allows us to conclude that $\nrho$ satisfies $\nrho(\tend,x)=\rho_1(x)$.

For the second claim, let
    \[
        \psi(t,x)=-\frac{1}{2}x'\Pi_0(t)x+m(t)'x+c(t)
    \]
with
    \[
        c(t)=-\frac{1}{2}\int_{\tstart}^{t}m(\tau)'B(\tau)B(\tau)'m(\tau)d\tau.
    \]
Clearly, $u(t,x)=B(t)'\nabla\psi$ while
    \begin{eqnarray*}
        &&\frac{\partial \psi}{\partial t}+A(t)x\cdot\nabla\psi+\frac{1}{2}\nabla\psi\cdot B(t)B(t)'\nabla\psi\\
        &=& -\frac{1}{2}x'\dot{\Pi}_0(t)x+\dot{m}(t)'x+\dot{c}(t)+x'A(t)'(-\Pi_0(t)x+m(t))\\
        &&+\frac{1}{2}(-x'\Pi_0(t)+m(t)') B(t)B(t)'(-\Pi_0(t)x+m(t))\\
        &=&\frac{1}{2}x'(A(t)'\Pi_0+\Pi_0 A(t)-\Pi_0(t)B(t)B(t)'\Pi_0(t))x-m(t)'(A(t)-B(t)B(t)'\Pi_0(t))x+\dot{c}(t)\\
        &&+x'A(t)'(-\Pi_0(t)x+m(t))+\frac{1}{2}(-x'\Pi_0(t)+m(t)') B(t)B(t)'(-\Pi_0(t)x+m(t))\\
        &=& \dot{c}(t)+\frac{1}{2}m(t)'B(t)B(t)'m(t)=0.
    \end{eqnarray*}
\end{proof}

\section{Numerical examples}\label{sec:examples}
We present two examples. The first one is on steering a collection of {\em inertial} particles in a $2$-dimensional phase space between Gaussian marginal distributions at the two end-points of a time interval. We use the closed-form control presented in Section \ref{sec:gaussian}. The second example is on steering distributions in a one-dimensional state-space with specified prior dynamics and more general marginal distributions. In both examples, we observe that the entropic interpolations converge to the displacement interpolation as the diffusion coefficient goes to zero.

\subsection{Gaussian marginals}
Consider a large collection of inertial particles moving in a $1$-dimension configuration space (i.e., for each particle, the position $x(t)\in{\mathbb R}$). The position $x$ and velocity $v$ of particles are assumed to be jointly normally distributed in the $2$-dimensional phase space ($(x,v)\in {\mathbb R}^2$) with mean and variance
	\[
		m_0=\left[\begin{matrix}-5 \\-5\end{matrix}\right],\mbox{ and }
        \Sigma_0=\left[\begin{matrix}1 & 0\\ 0 & 1\end{matrix}\right]
	\]
at $t=0$. We seek to steer the particles to a new joint normal distribution with mean and variance
	\[
		m_1=\left[\begin{matrix}5 \\5\end{matrix}\right],\mbox{ and }
		\Sigma_1=\left[\begin{matrix}1 & 0\\ 0 & 1\end{matrix}\right]
	\]
at $t=1$. The problem to steer the particles provides also a natural way to interpolate these two end-point marginals by providing a flow of one-time marginals at intermediary points $t\in[0,1]$.

When the particles experience stochastic forcing, their trajectories correspond to a Schr\"{o}dinger bridge with reference evolution
	\[
        d\left(\begin{array}{c}x(t)\\v(t)\end{array}\right)=
        \left[\begin{array}{cc}0 & 1\\ 0 & 0 \end{array}\right]
        \left(\begin{array}{c}x(t)\\v(t) \end{array}\right) dt+\left[\begin{array}{c}0\\1 \end{array}\right]\sqrt{\epsilon}dw(t).
	\]
In particular, we are interested in the behavior of trajectories when the random forcing is negligible compared to the ``deterministic'' drift.

Figure \ref{fig:interpolation1} depicts the flow of the one-time marginals of the Schr\"{o}dinger bridge with $\epsilon=9$. The transparent tube represents the $3\sigma$ region
	\[
		(\xi(t)'-m_t')\Sigma_t ^{-1}(\xi(t)-m_t) \le 9,~~~\xi(t)=
        \left[\begin{array}{c}x(t)\\v(t) \end{array}\right]
	\]
and the curves with different color stand for typical sample paths of the Schr\"{o}dinger bridge. Similarly, Figures \ref{fig:interpolation2} and \ref{fig:interpolation3} depict the corresponding flows for $\epsilon=4$ and $\epsilon=0.01$, respectively. The interpolating flow in the absence of stochastic disturbance, i.e., for the optimal transport with prior, is depicted in Figure \ref{fig:interpolation4}; the sample paths are now smooth as compared to the corresponding sample paths with stochastic disturbance.
As $\epsilon\searrow 0$, the paths converge to those corresponding to optimal transport and $\epsilon=0$. For comparison, we also provide in Figure \ref{fig:interpolation5} the interpolation corresponding to optimal transport without prior, i.e., for the trivial dynamics $A(t)\equiv 0$ and $B(t)\equiv I$, which is precisely a constant speed translation.
\begin{figure}\begin{center}
    \includegraphics[width=0.50\textwidth]{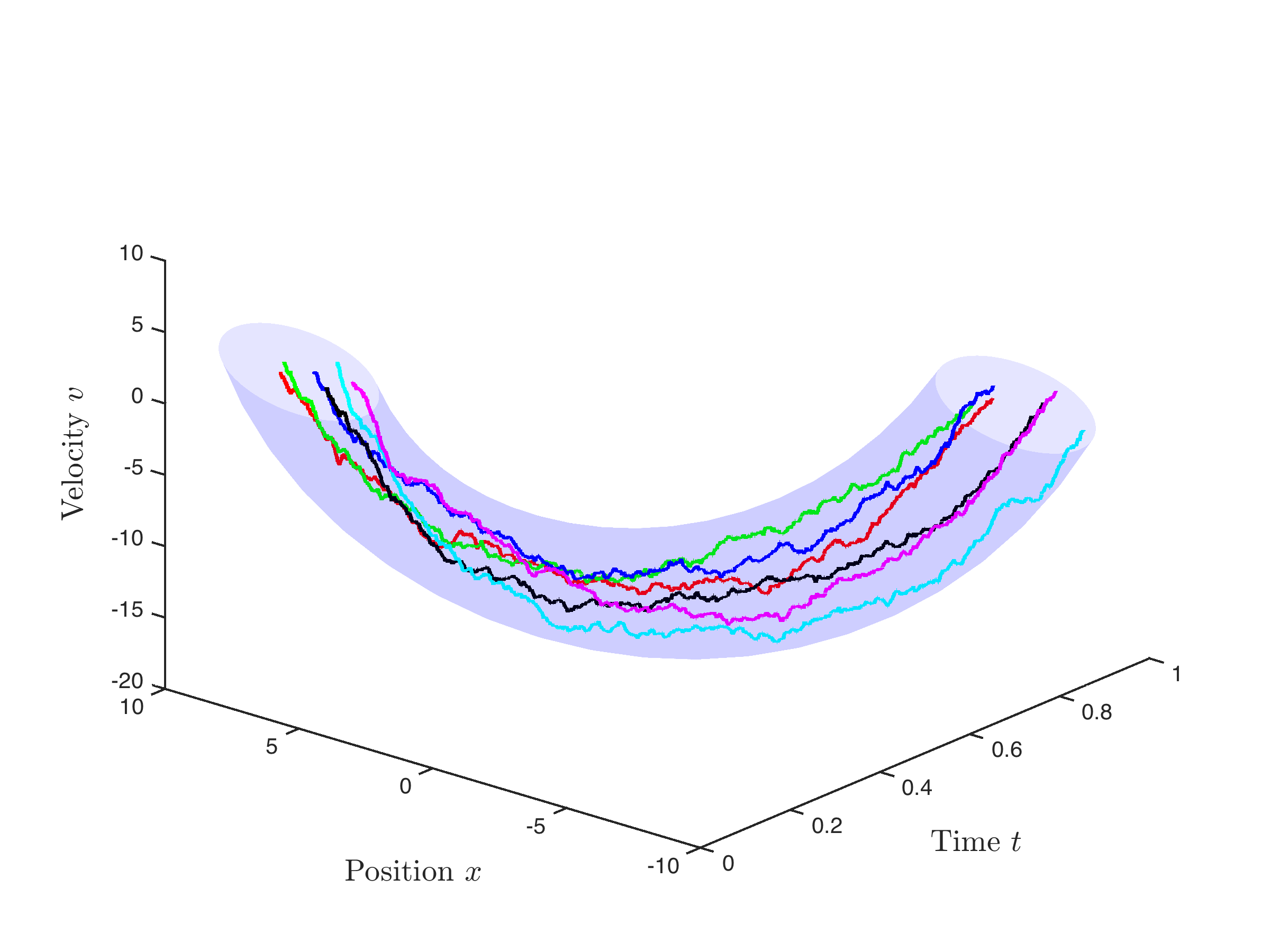}
    \caption{Interpolation based on Schr\"{o}dinger bridge with $\epsilon=9$}
    \label{fig:interpolation1}
\end{center}\end{figure}
\begin{figure}\begin{center}
    \includegraphics[width=0.50\textwidth]{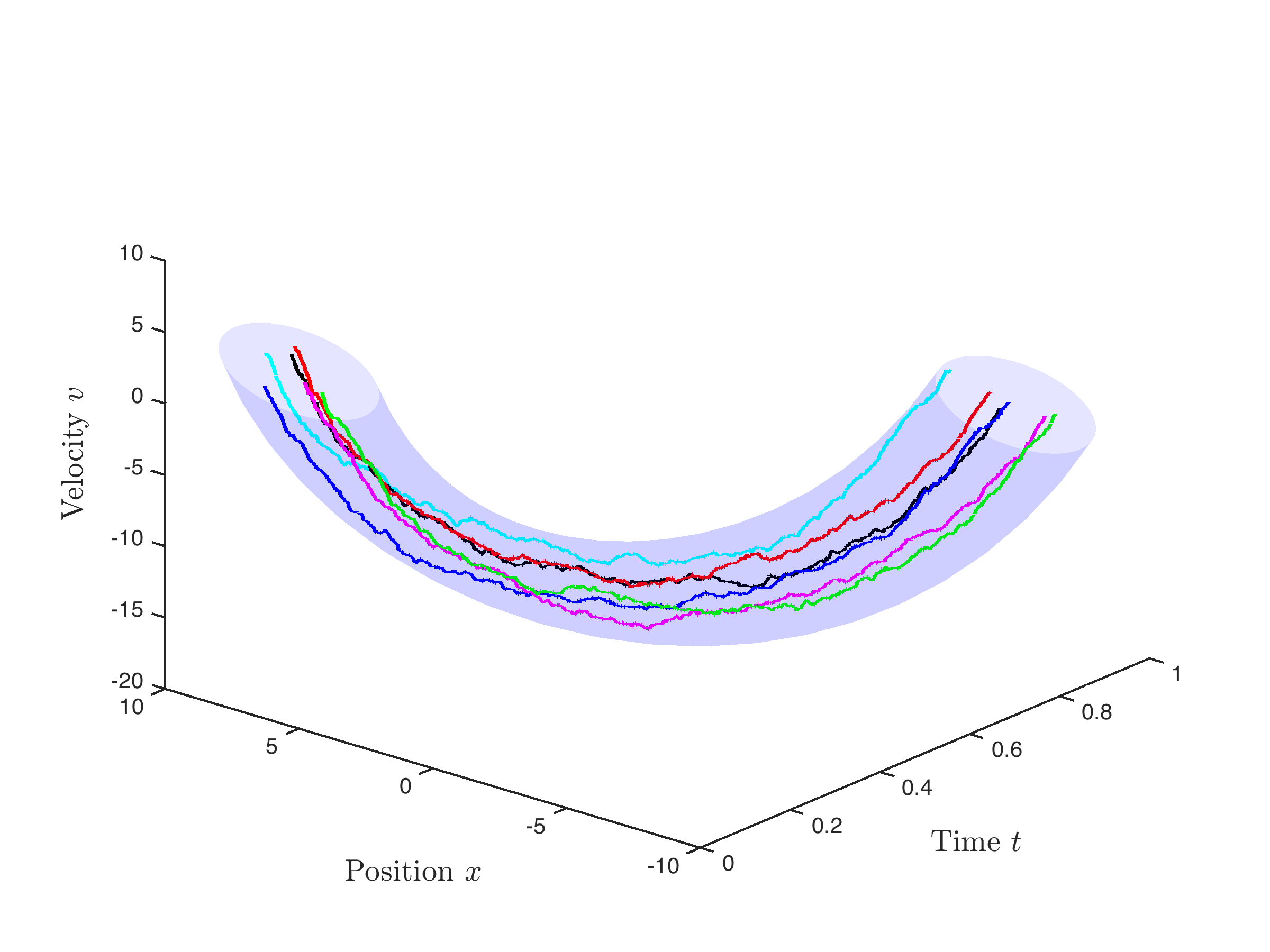}
    \caption{Interpolation based on Schr\"{o}dinger bridge with $\epsilon=4$}
    \label{fig:interpolation2}
\end{center}\end{figure}
\begin{figure}\begin{center}
    \includegraphics[width=0.50\textwidth]{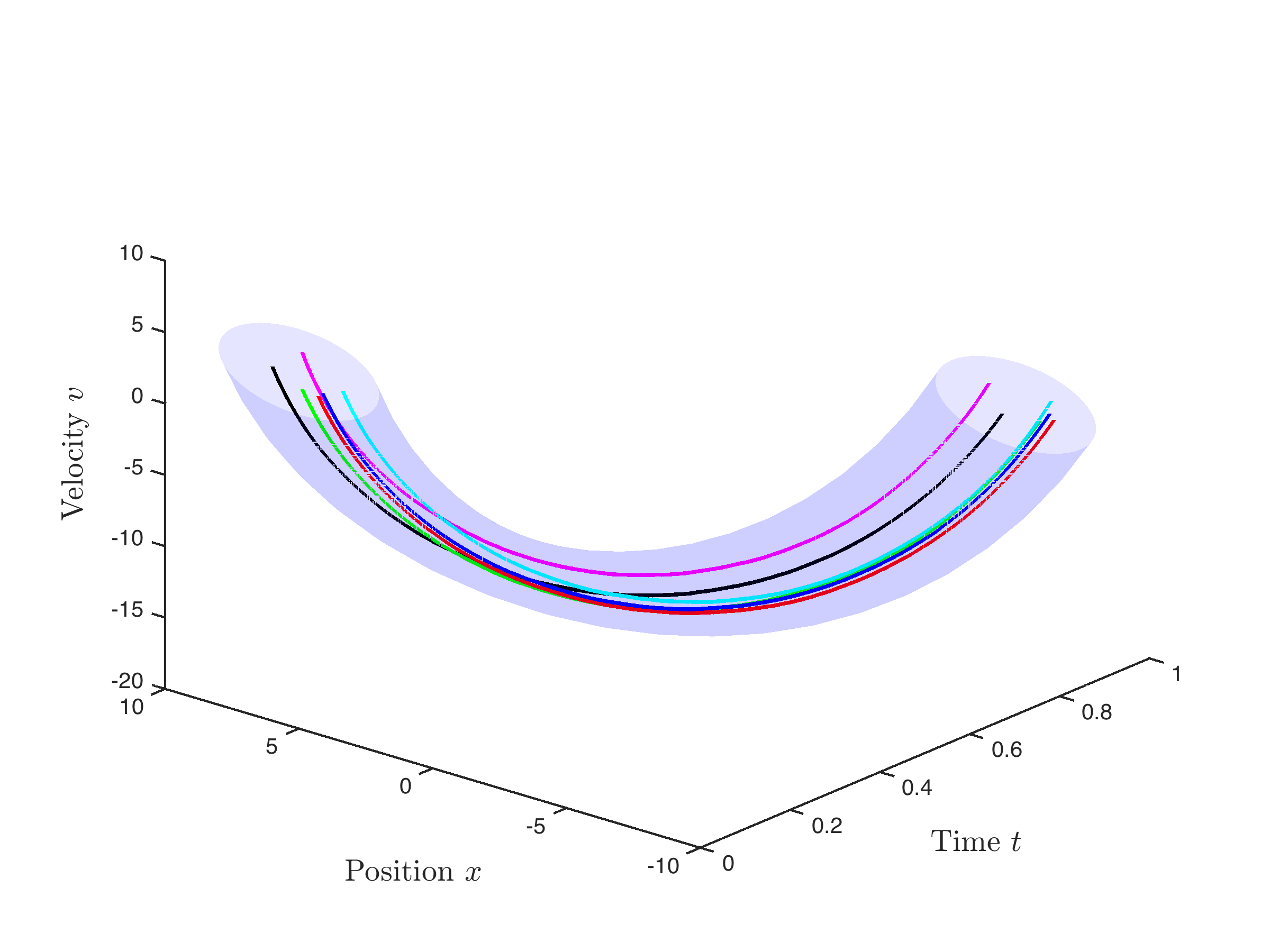}
    \caption{Interpolation based on Schr\"{o}dinger bridge with $\epsilon=0.01$}
    \label{fig:interpolation3}
\end{center}\end{figure}
\begin{figure}\begin{center}
    \includegraphics[width=0.50\textwidth]{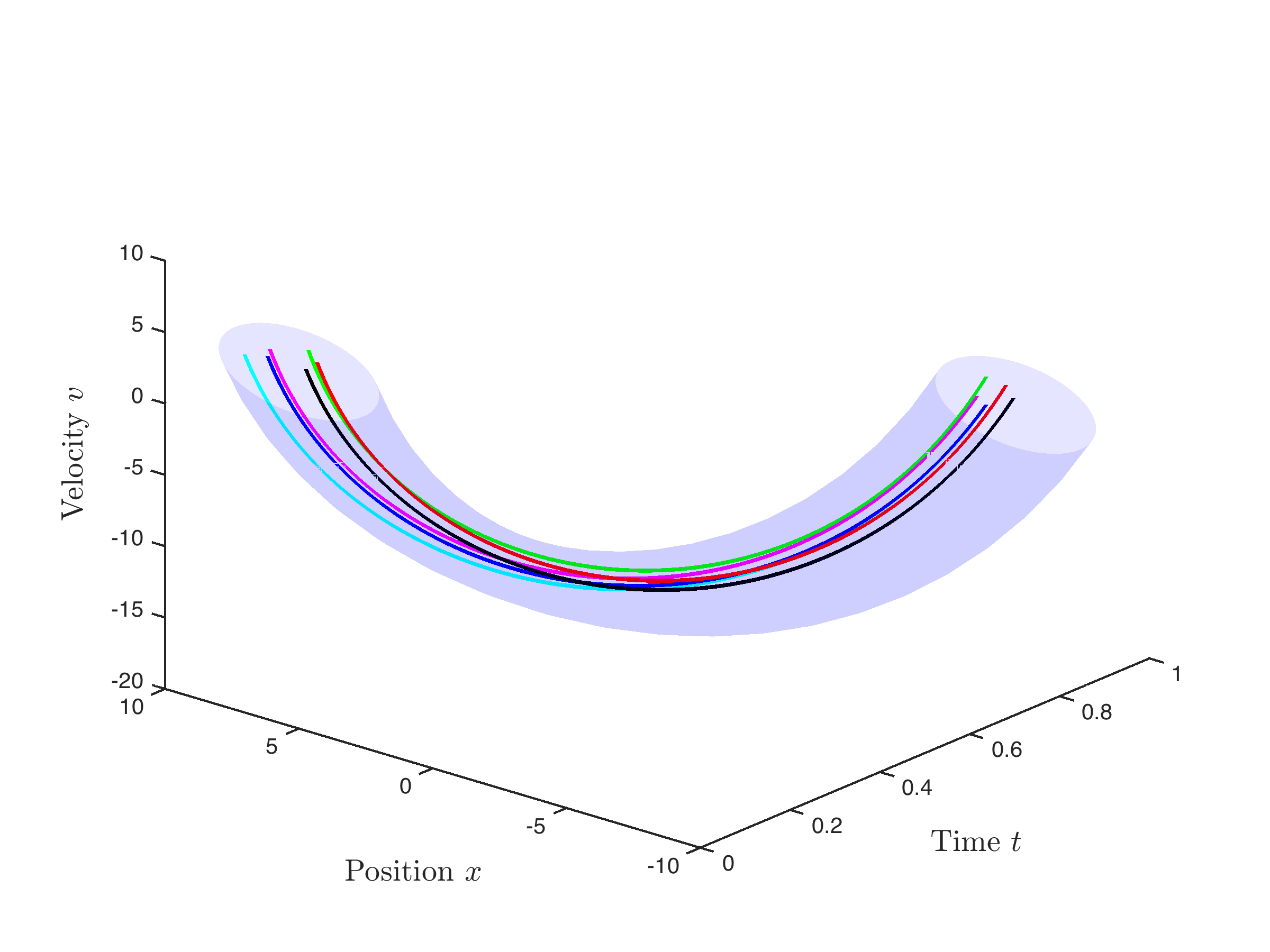}
    \caption{Interpolation based on OMT-wpd}
    \label{fig:interpolation4}
\end{center}\end{figure}
\begin{figure}\begin{center}
    \includegraphics[width=0.50\textwidth]{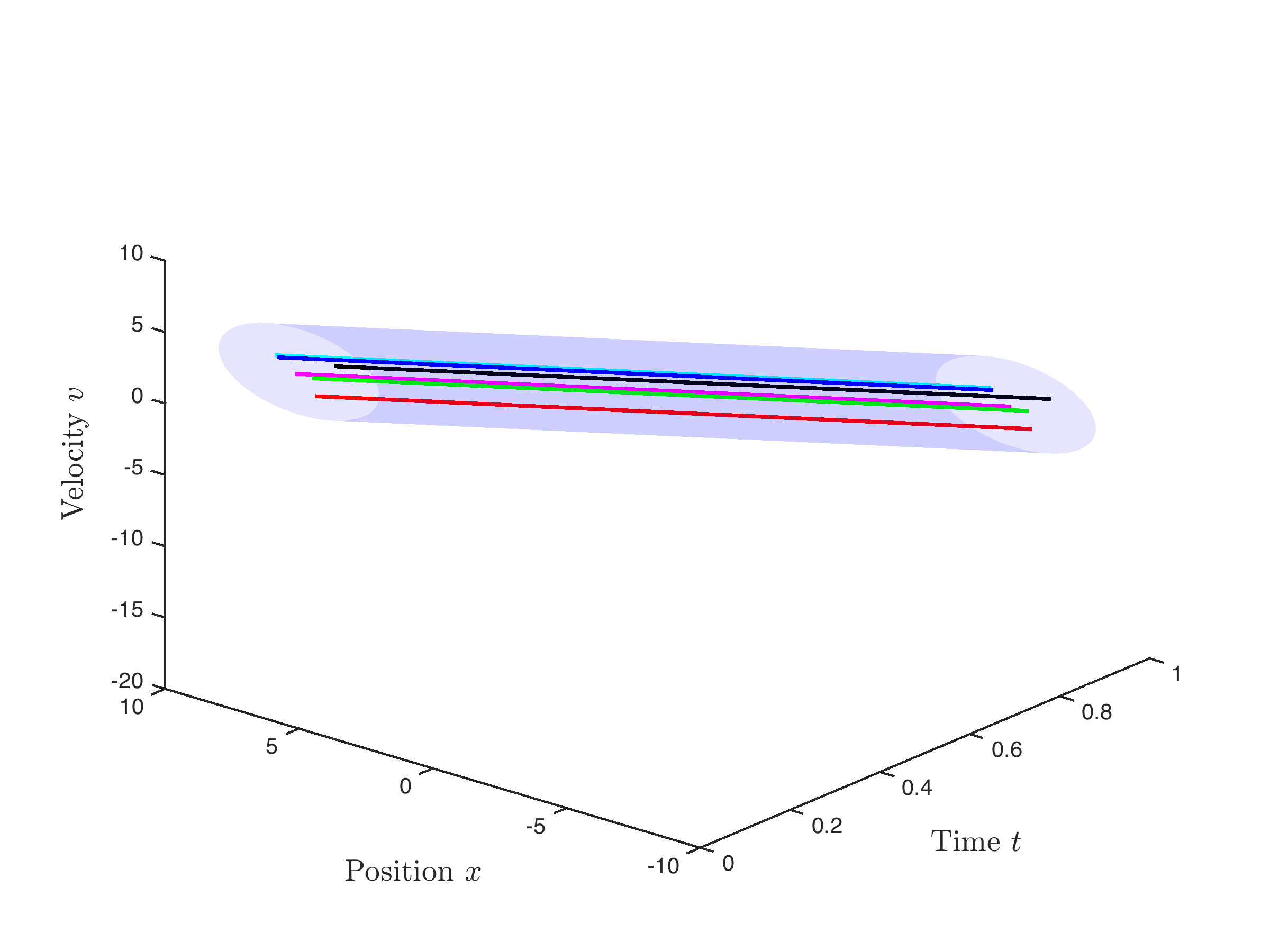}
    \caption{Interpolation based on OMT}
    \label{fig:interpolation5}
\end{center}\end{figure}

\subsection{General marginals}
Consider now a large collection of particles obeying
	\[
		dx(t)=-2x(t)dt+u(t)dt
	\]
in $1$-dimensional state space with marginal distributions
	\[
		\rho_0(x) =
        \begin{cases}
        {0.2-0.2\cos(3\pi x)+0.2} & \text{if}~ 0\le x<2/3\\
        {5-5\cos(6\pi x-4\pi)+0.2} & \text{if}~ 2/3\le x\le 1,
        \end{cases}
	\]
and
    \[
        \rho_1(x)=\rho_0(1-x).
    \]
These are shown in Figure \ref{fig:marginals} and, obviously, are not Gaussian.
Once again, our goal is to steer the state of the system (equivalently, the particles) from the initial distribution $\rho_0$ to the final $\rho_1$ using minimum energy control. That is, we need to solve the problem of OMT-wpd.
In this $1$-dimensional case, just like in the classical OMT problem, the optimal transport map $y=T(x)$ between the two end-points can be determined from\footnote{
In this 1-dimensional case, \eqref{eq:monotoneforoneD} is a simple rescaling and, therefore, $T(\cdot)$ inherits the monotonicity of $\hat T(\cdot)$.}
    \[
        \int_{-\infty}^x \rho_0(y)dy=\int_{-\infty}^{T(x)}\rho_1(y)dy.
    \]
The interpolation flow $\rho_t,~0\le t\le 1$ can then be obtained using \eqref{eq:priorinterpolation}.  Figure \ref{fig:omtwithprior} depicts the solution of OMT-wpd. For comparison, we also show the solution of the classical OMT in figure \ref{fig:omt} where the particles move on straight lines.

Finally, we assume a stochastic disturbance,
	\[
		dx(t)=-2x(t)dt+u(t)dt+\sqrt\epsilon dw(t),
	\]
with $\epsilon>0$.
Figure \ref{fig:sbprior1}--\ref{fig:sbprior5} depict minimum energy flows for diffusion coefficients $\sqrt\epsilon=0.5,~0.3,~0.15,~0.05,~0.01$, respectively. As $\epsilon\to 0$, it is seen that the solution to the Schr\"odinger problem converges to the solution of the problem of OMT-wpd as expected.
\begin{figure}\begin{center}
    \includegraphics[width=0.50\textwidth]{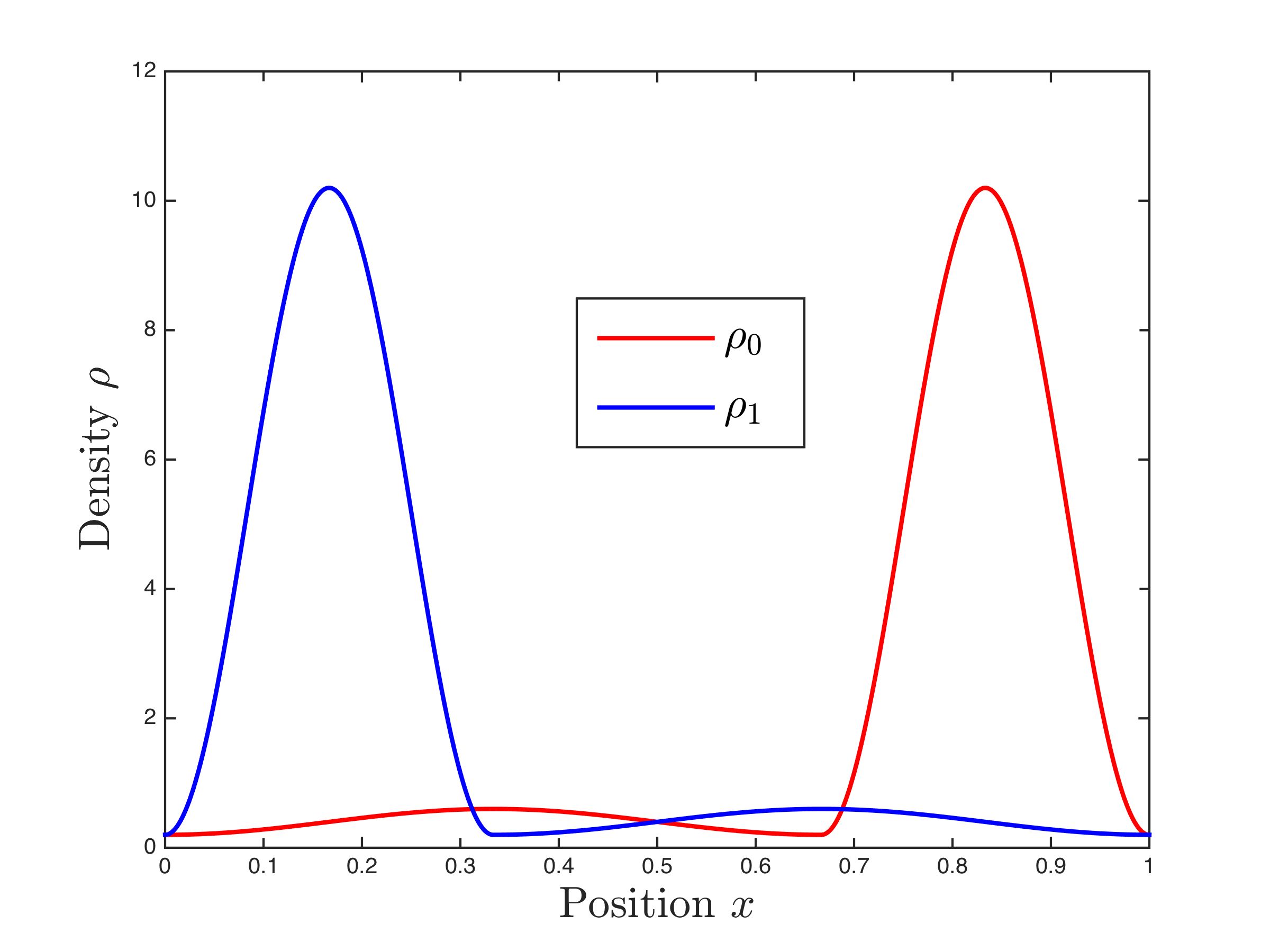}
    \caption{Marginal distributions}
    \label{fig:marginals}
\end{center}\end{figure}
\begin{figure}\begin{center}
    \includegraphics[width=0.50\textwidth]{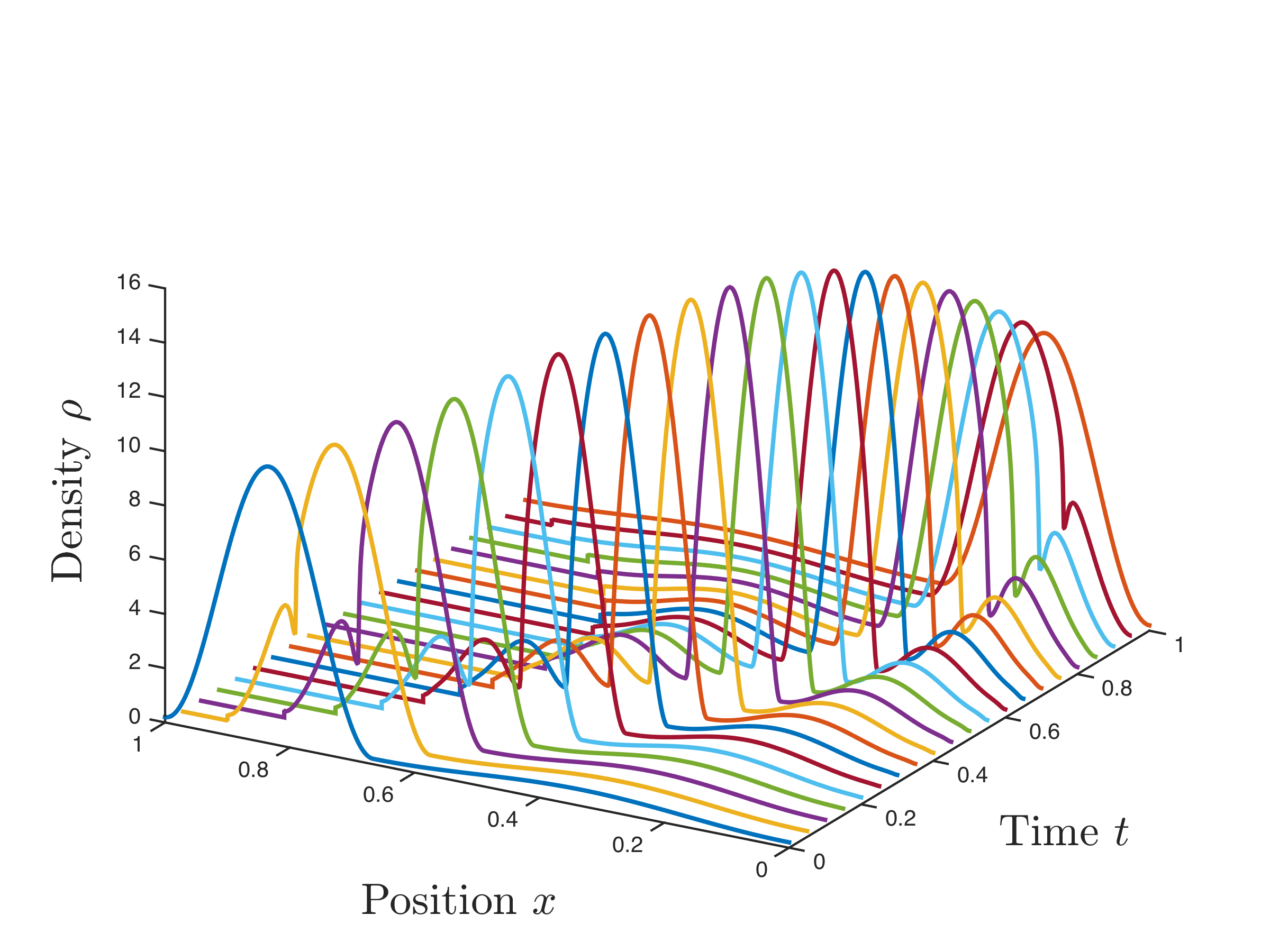}
    \caption{Interpolation based on OMT-wpd}
    \label{fig:omtwithprior}
\end{center}\end{figure}
\begin{figure}\begin{center}
    \includegraphics[width=0.50\textwidth]{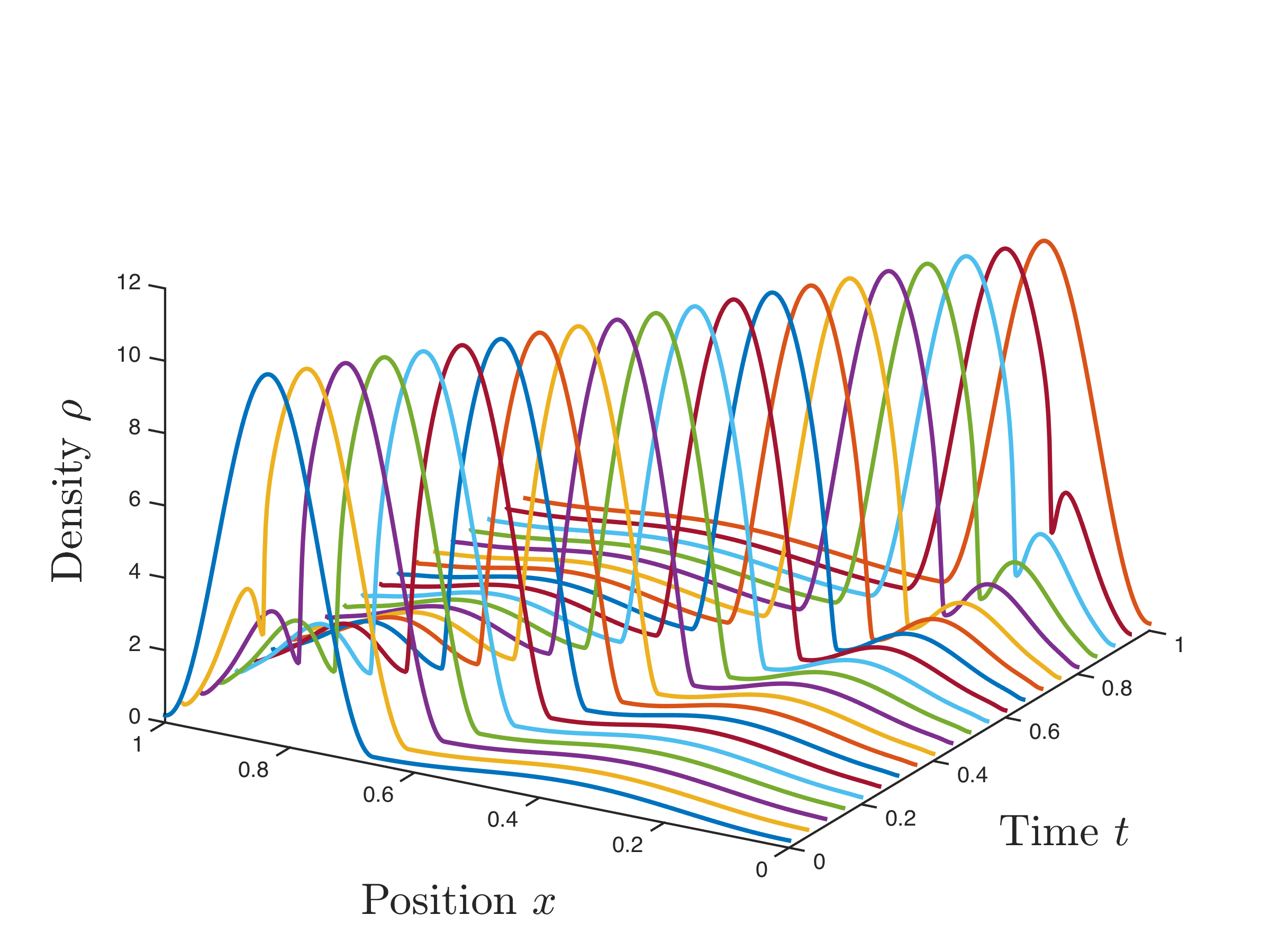}
    \caption{Interpolation based on OMT}
    \label{fig:omt}
\end{center}\end{figure}
\begin{figure}\begin{center}
    \includegraphics[width=0.50\textwidth]{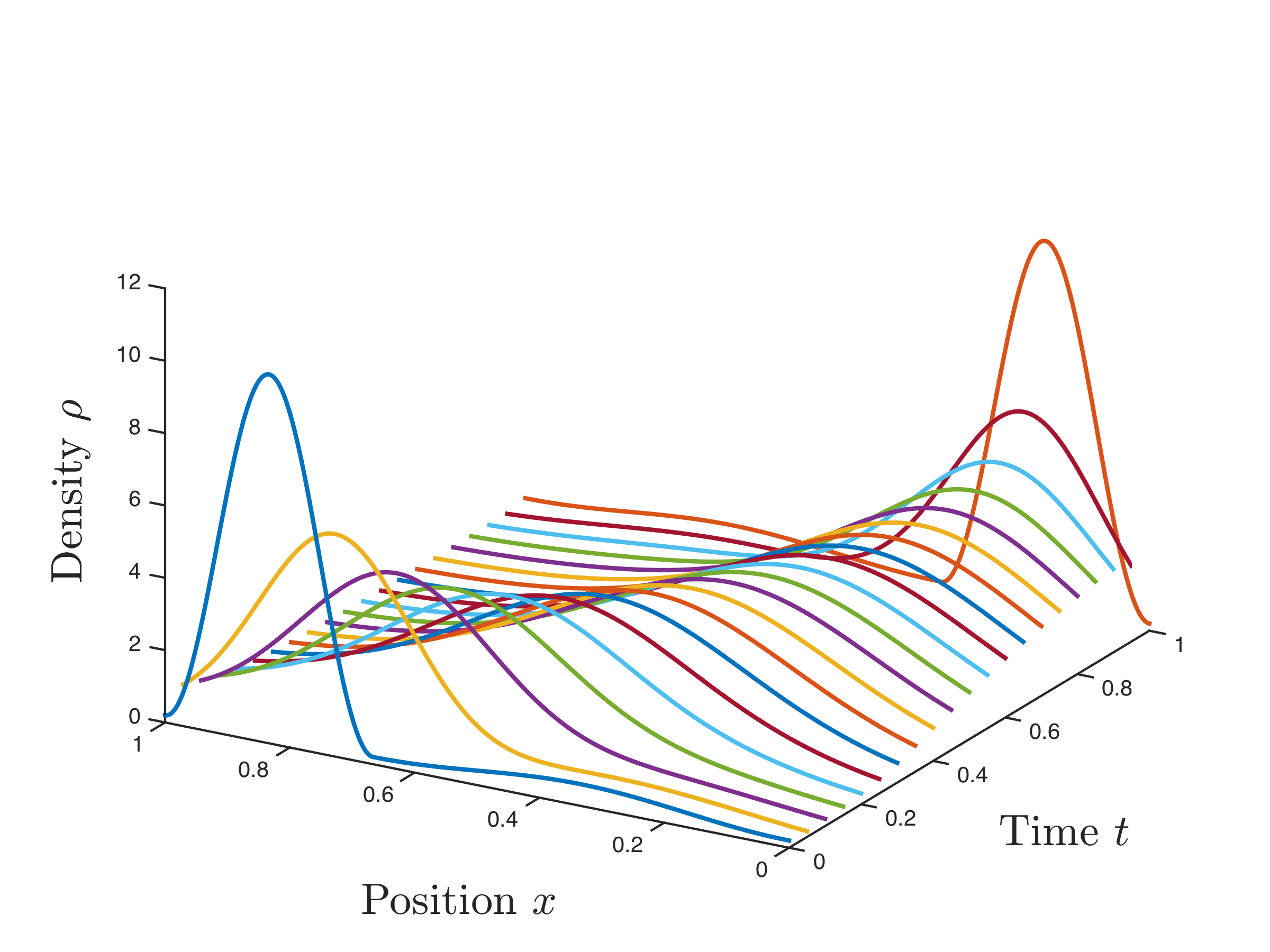}
    \caption{Interpolation based on Schr\"odinger bridge with $\sqrt\epsilon=0.5$}
    \label{fig:sbprior1}
\end{center}\end{figure}
\begin{figure}\begin{center}
    \includegraphics[width=0.50\textwidth]{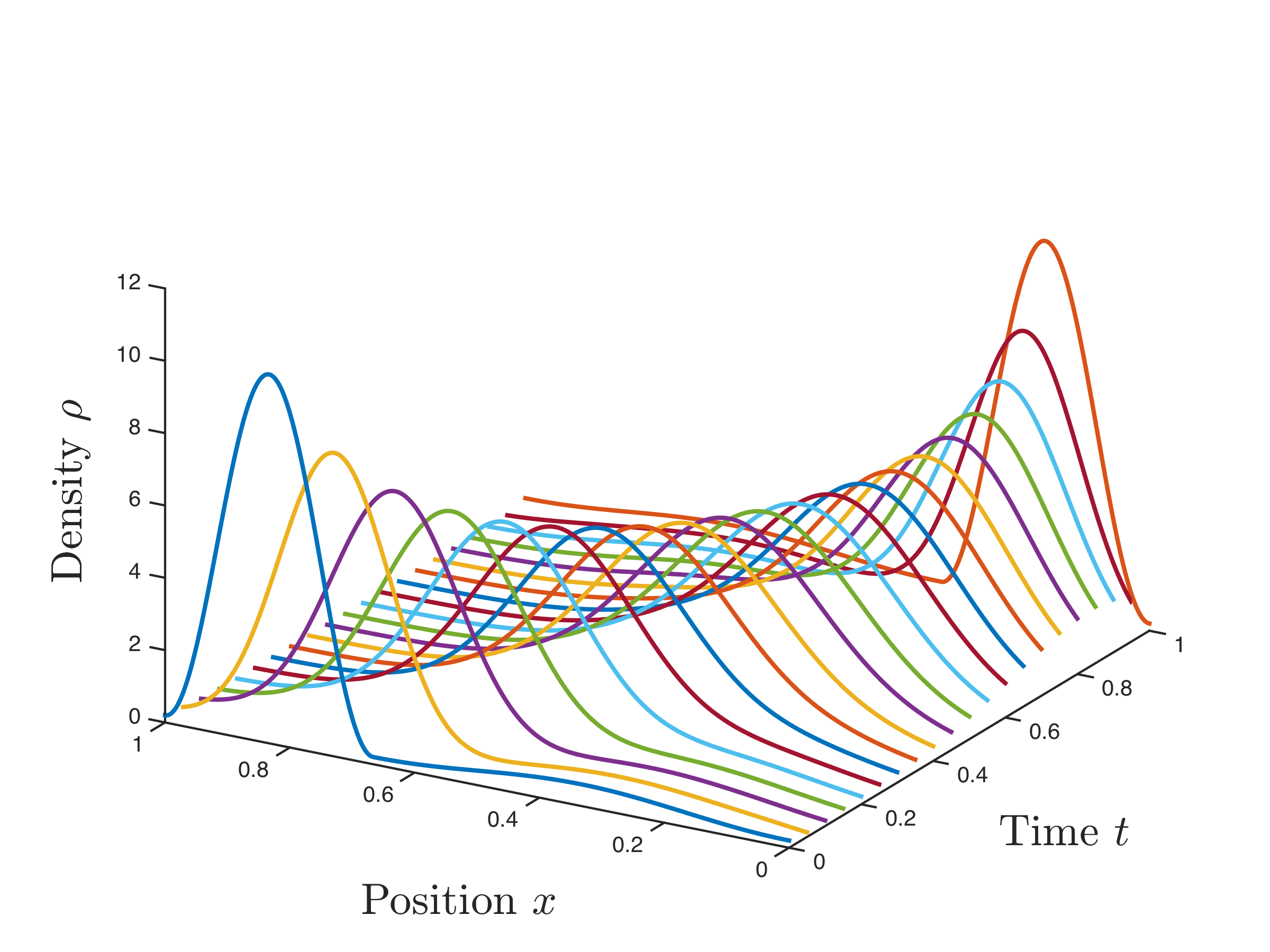}
    \caption{Interpolation based on Schr\"odinger bridge with $\sqrt\epsilon=0.3$}
    \label{fig:sbprior2}
\end{center}\end{figure}
\begin{figure}\begin{center}
    \includegraphics[width=0.50\textwidth]{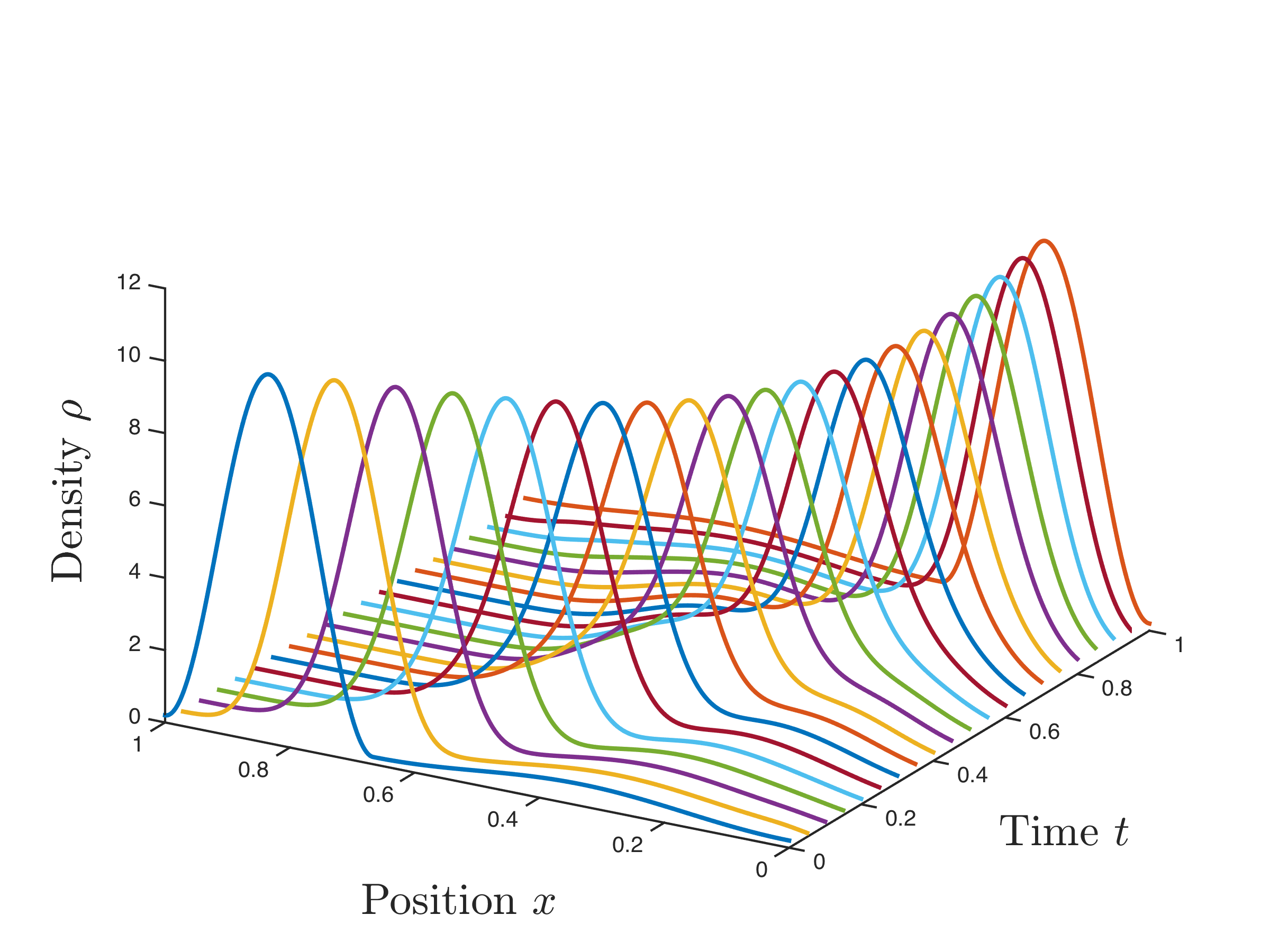}
    \caption{Interpolation based on Schr\"odinger bridge with $\sqrt\epsilon=0.15$}
    \label{fig:sbprior3}
\end{center}\end{figure}
\begin{figure}\begin{center}
    \includegraphics[width=0.50\textwidth]{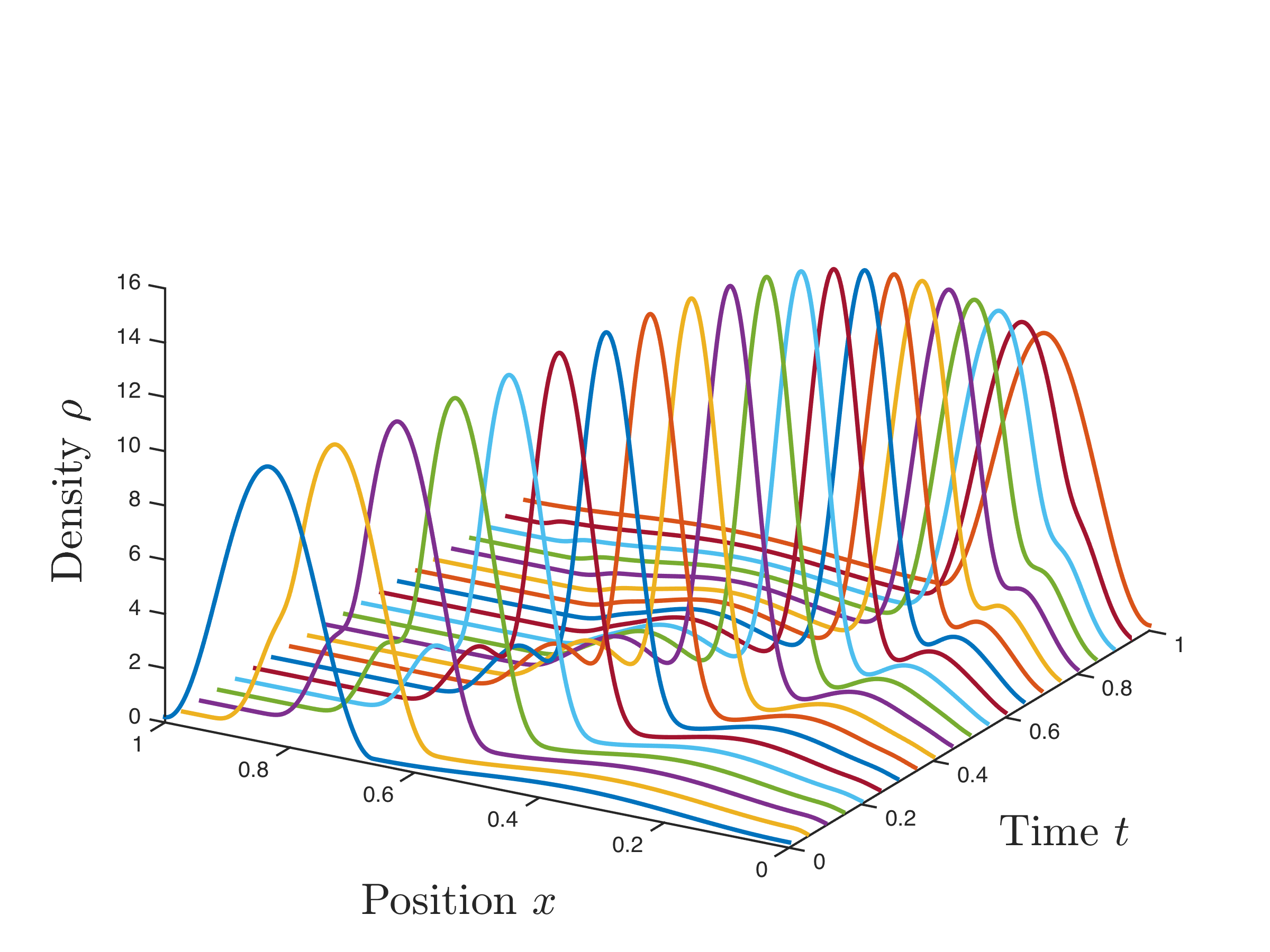}
    \caption{Interpolation based on Schr\"odinger bridge with $\sqrt\epsilon=0.05$}
    \label{fig:sbprior4}
\end{center}\end{figure}
\begin{figure}\begin{center}
    \includegraphics[width=0.50\textwidth]{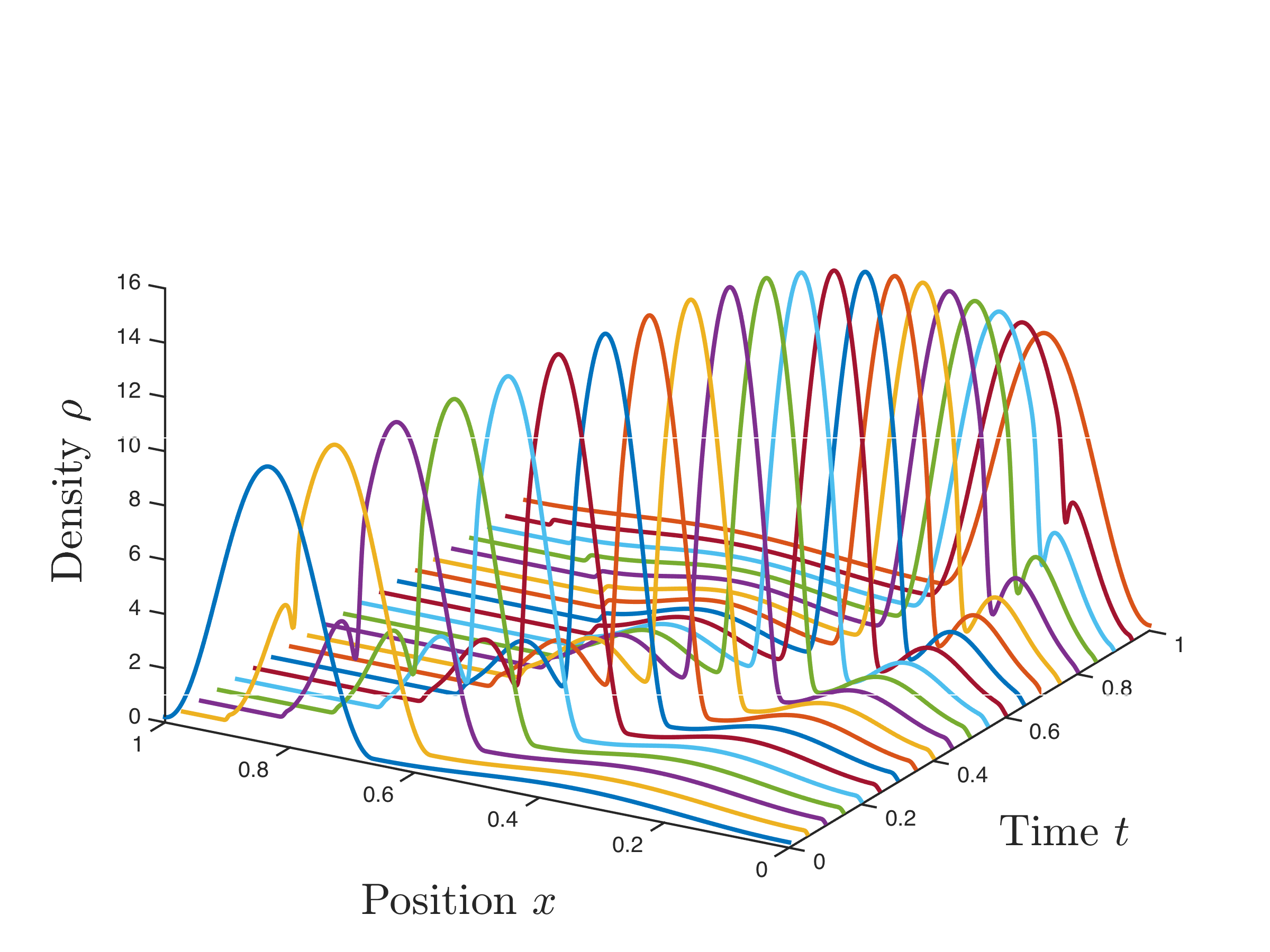}
    \caption{Interpolation based on Schr\"odinger bridge with $\sqrt\epsilon=0.01$}
    \label{fig:sbprior5}
\end{center}\end{figure}

\section{Recap}\label{sec:conclusion}
The problem to steer the random state of a dynamical system between given probability distributions can be equally well be seen as the control problem to simultaneously herd a collection of particles obeying the given dynamics, or as the problem to identify a potential that effects such a transition.
The former is seen to have applications in the control of uncertain systems, system of particles, etc. The latter is seen as a modeling problem and system identification problem, where e.g., the collective response of particles is observed and the prior dynamics need to be adjusted by postulating a suitable potential so as to be consistent with observed marginals.
When the dynamics are trivial and the state matrix is zero while the input matrix is the identity, the problem reduces to the classical OMT problem. Herein we presented a generalization to nontrivial linear dynamics.
A version of both viewpoints where an added stochastic disturbance is present relates to the problem of constructing the so-called Schr\"odinger bridge between two end-point marginals. In fact, Schr\"odinger's bridge problem was conceived as a modeling problem to identify a probability law on path space that is closest to a prior (usually a Wiener measure) and is consistent with the marginals. Its stochastic control reformulation in the 90's has led to a rapidly developing subject. The present work relates OMT as a limit to Schr\"odinger bridges, when the stochastic disturbance goes to zero, and discusses the generalization of both to the setting where the prior linear dynamics are quite general. It opens the way to employ the efficient iterative techniques recently developed for Schr\"odinger bridges to the computationally challenging OMT (with or without prior dynamics). This is the topic of \cite{CheGeoPav15a}.

\section*{Appendix: Proof of Theorem \ref{thm:priorslowingdown}}\label{sec:appendix}

Let $q^\epsilon$ be the Markov kernel of \eqref{eq:priorprocess}, then
    \[
        q^\epsilon(s,x,t,y)=(2\pi\epsilon)^{-n/2}|M(t,s)|^{-1/2}\exp\left(-\frac{1}{2\epsilon}(y-\Phi(t,s)x)'M(t,s)^{-1}(y-\Phi(t,s)x)\right).
    \]
Comparing this and the Brownian kernel $q^{B,\epsilon}$ we obtain
    \[
        q^\epsilon(s,x,t,y)=(t-s)^{n/2}|M(t,s)|^{-1/2}q^{B,\epsilon}(s,M(t,s)^{-1/2}\Phi(t,s)x,t,M(t,s)^{-1/2}y).
    \]
Now define two new marginal distributions $\hat\rho_0$ and $\hat\rho_1$ through the coordinates transformation $C$ in \eqref{eq:coordtrans}, 
    \begin{eqnarray*}
        \hat\rho_0(x)&=&|M_{10}|^{1/2}|\Phi_{10}|^{-1}\rho_0(\Phi_{10}^{-1}M_{10}^{1/2}x)\\
        \hat\rho_1(x)&=&|M_{10}|^{1/2}\rho_1(M_{10}^{1/2}x).
    \end{eqnarray*}
Let $(\hat\varphi_0, \varphi_1)$ be a pair that solves the Schr\"odinger bridge problem with kernel $q^\epsilon$ and marginals $\rho_0, \rho_1$, and define $(\hat\varphi_0^B,\varphi_1^B)$ as
    \begin{subequations}\label{eq:relationschrodinger}
    \begin{eqnarray}
        \hat\varphi_0(x)&=&|\Phi_{10}|\hat\varphi_0^B(M_{10}^{-1/2}\Phi_{10}x)\\
        \varphi_1(x)&=&|M_{10}|^{-1/2}\varphi_1^B(M_{10}^{-1/2}x),
    \end{eqnarray}
    \end{subequations}
then the pair $(\hat\varphi_0^B,\varphi_1^B)$ solves the Schr\"odinger bridge problem with kernel $q^{B,\epsilon}$ and marginals $\hat\rho_0, \hat\rho_1$. To verify this, we need only to show that the joint distribution
    \[
        \cP^{B,\epsilon}_{01}(E)=\int_E q^{B,\epsilon}(0,x,1,y)\hat\varphi_0^B(x)\varphi_1^B(y)dxdy
    \]
matches the marginals $\hat\rho_0, \hat\rho_1$. This follows from
    \begin{eqnarray*}
        \int_{\mR^n} q^{B,\epsilon}(0,x,1,y)\hat\varphi_0^B(x)\varphi_1^B(y)dy
        &=& 
        \int_{\mR^n} q^{B,\epsilon}(0,x,1,M_{10}^{-1/2}y)\hat\varphi_0^B(x)\varphi_1^B(M_{10}^{-1/2}y)d(M_{10}^{-1/2}y)
        \\&=&
        |M_{10}|^{1/2}|\Phi_{10}|^{-1}\int_{\mR^n} q^\epsilon(0,\Phi_{10}^{-1}M_{10}^{1/2}x,1,y)\hat\varphi_0(\Phi_{10}^{-1}M_{10}^{1/2}x)\varphi_1(y)dy
        \\&=&
        |M_{10}|^{1/2}|\Phi_{10}|^{-1} \rho_0(\Phi_{10}^{-1}M_{10}^{1/2}x)=\hat\rho_0(x),
    \end{eqnarray*}
and
    \begin{eqnarray*}
        \int_{\mR^n} q^{B,\epsilon}(0,x,1,y)\hat\varphi_0^B(x)\varphi_1^B(y)dx
        &=& 
        \int_{\mR^n} q^{B,\epsilon}(0,M_{10}^{-1/2}\Phi_{10}x,1,y)\hat\varphi_0^B(M_{10}^{-1/2}\Phi_{10}x)\varphi_1^B(y)d(M_{10}^{-1/2}\Phi_{10}x)
        \\&=&
        |M_{10}|^{1/2}\int_{\mR^n} q^\epsilon(0,x,1,M_{10}^{1/2}y)\hat\varphi_0(x)\varphi_1(M_{10}^{1/2}y)dx
        \\&=&
        |M_{10}|^{1/2} \rho_1(M_{10}^{1/2}y)=\hat\rho_1(y).
    \end{eqnarray*}
Compare $\cP_{01}^{B,\epsilon}$ with $\cP_{01}^{\epsilon}$ it is not difficult to find out that $\cP_{01}^{B,\epsilon}$ is a push-forward of $\cP_{01}^{\epsilon}$, that is,
    \[
        \cP_{01}^{B,\epsilon}=C_\sharp \cP_{01}^{\epsilon}.
    \]
On the other hand, let $\pi^B$ be the solution to classical OMT \eqref{eq:OptTrans} with marginals $\hat\rho_0, \hat\rho_1$, then
    \[
        \pi^B=C_\sharp \pi.
    \]
Now since $\cP_{01}^{B,\epsilon}$ weakly converge to $\pi^B$ from Theorem \ref{thm:slowingdown}, we conclude that $\cP_{01}^{\epsilon}$ weakly converge to $\pi$ as $\epsilon$ goes to 0.

We next show $\cP_{t}^{\epsilon}$ weakly converges to $\mu_t$ as $\epsilon$ goes to 0. The displacement interpolation $\mu$ can be decomposed as
    \begin{eqnarray*}
        \mu(\cdot)=\int_{\mR^n\times\mR^n} \delta_{\gamma^{xy}}(\cdot) ~d\pi(x,y),
    \end{eqnarray*}
where $\gamma^{xy}$ is the minimum energy path \eqref{eq:priorgeodesic} connecting $x, y$, and $\delta_{\gamma^{xy}}$ is the Dirac measure at $\gamma^{xy}$ on the path space. Similarly, the entropic interpolation $\cP^\epsilon$ can be decomposed as
    \begin{eqnarray*}
        \cP^\epsilon(\cdot)=\int_{\mR^n\times\mR^n} \cQ_{xy}^{\epsilon}(\cdot)~d\cP_{01}^{\epsilon}(x,y),
    \end{eqnarray*}
where $\cQ_{xy}^{\epsilon}$ is the pinned bridge \cite{CheGeo14} associated with \eqref{eq:priorprocess} conditioned on $x(0)=x$ and $x(1)=y$. It has the stochastic differential equation representation
    \[
        dx(t)=(A(t)-B(t)B(t)'\Phi(1,t)'M(1,t)^{-1}\Phi(1,t))x(t)dt
        +B(t)B(t)'\Phi(1,t)'M(1,t)^{-1}ydt+\sqrt{\epsilon}B(t)dw(t).
    \]
As $\epsilon$ goes to zero, it converges to
    \[
        dx(t)=(A(t)-B(t)B(t)'\Phi(1,t)'M(1,t)^{-1}\Phi(1,t))x(t)dt
        +B(t)B(t)'\Phi(1,t)'M(1,t)^{-1}ydt,~x(0)=x,
    \]
which is $\gamma^{xy}$. In other word, $\cQ_{xy}^{\epsilon}$ weakly converges to $\delta_{\gamma^{xy}}$. This together with the fact that $\cP_{01}^{\epsilon}$ weakly converges to $\pi$ show that $\cP_{t}^{\epsilon}$ weakly converges to $\mu_t$ as $\epsilon$ goes to 0.

\spacingset{.97}
\bibliographystyle{IEEEtran}
\bibliography{refs}
\end{document}